\documentclass[12pt]{amsart}
\usepackage{amsmath,amssymb, xypic,verbatim,amscd}
\newcommand\bull{\sssize{\bullet}}

\newcommand\coo{\operatorname{co}}
\newcommand\OMM{\Omega}
\newcommand\OMC{\bar{\Omega}}
\newcommand\LAM{\Phi}
\newcommand\lam{\lambda}

\newcommand\ann{\operatorname{ann}}

\newcommand\FL{\operatorname{FL}}
\newcommand\frk{\mathfrak{k}}
\newcommand\EV{\Gamma}
\newcommand\mo{\mathcal{O}}
\newcommand\frh{\mathfrak{h}}
\newcommand\frb{\mathfrak{b}}

\newcommand\frg{\mathfrak{g}}

\newcommand\mh{\mathcal{H}}
\newcommand\Spin{\operatorname{Spin}}

\newcommand\home{\operatorname{Hom}}

\newcommand\mf{\mathcal{F}}
\newcommand\mk{\mathcal{K}}

\newcommand\SU{SU}

\def\simto{\overset{\sim}{\longrightarrow}}
\newcommand\me{\mathcal{E}}

\newcommand\Fl{\operatorname{Fl}}
\newcommand\Sym{\operatorname{sym}}
\newcommand\SL{\operatorname{SL}}
\newcommand\Sp{\operatorname{Sp}}
\newcommand\sym{\operatorname{sym}}

\newcommand\mv{\mathcal{V}}
\newcommand\mg{\mathcal{G}}
\newcommand\sm{B}
\newcommand\ssm{b}

\newcommand\cosym{\operatorname{cosym}}

\newcommand\tensor{\otimes}

\newcommand\ml{\mathcal{L}}
\newcommand\OG{\operatorname{OG}}

\newcommand{\leto}[1]{\stackrel{#1}{\to}}
\newcommand\ma{\mathcal{A}}

\newcommand\codim{\operatorname{codim}}
\newcommand\mb{\mathcal{B}}
\newcommand\co{\operatorname{co}}

\newcommand\FS{\mathfrak{S}}
\newcommand\eps{\epsilon}

\newcommand\im{\operatorname{im}}
\newcommand\IFl{\operatorname{IFl}}
\newcommand\mw{\mathcal{W}}
\newcommand\mc{\mathcal{C}}
\newcommand\mq{\mathcal{Q}}

\newcommand\Gl{\operatorname{GL}}

\newcommand\Gr{\operatorname{Gr}}
\newcommand\GL{\operatorname{Gl}}

\newcommand\SO{\operatorname{SO}}

\newcommand\IG{\operatorname{IG}}
\newcommand\ep{\epsilon}
\newtheorem{theorem}{Theorem}

\newtheorem{corollary}[theorem]{Corollary}
\newtheorem{conjecture}[theorem]{Conjecture}

\newtheorem{proposition}[theorem]{Proposition}

\newtheorem{lemma}[theorem]{Lemma}

\newtheorem{defi}[theorem]{Definition}

\newtheorem{remark}[theorem]{{ Remark}}

\begin{document}
\title[Eigencone and saturation]
{Eigencone, saturation and Horn
problems for symplectic and odd orthogonal groups}
 \author[Prakash Belkale and Shrawan Kumar]{Prakash Belkale and Shrawan Kumar\\Department of Mathematics\\
University of North Carolina\\
Chapel Hill, NC  27599--3250}
\maketitle
\section{Introduction}
In this paper we consider the eigenvalue problem, intersection
theory of homogenous spaces (in particular, the Horn problem) and the saturation problem for the
symplectic and odd orthogonal groups. The classical embeddings of
these groups in the special linear groups will play an important
role. We  deduce properties for these
classical groups from the known properties for the special linear
groups. The tangent space  techniques from
~\cite{gh}  play a crucial role. Another crucial ingredient is the
relationship between the intersection theory of the homogeneous
spaces for $\Sp(2n)$ and $\SO(2n+1)$.

Let $G$ be a semisimple, connected, complex algebraic group with a maximal
compact subgroup $K$.  Let $\frh_{+}$ be the positive Weyl chamber
of $G$. Then, there is a bijection $C:
\frh_{+} \to \frk/K$, $x\mapsto \overline{ix}$, where $\frk$ is the Lie algebra of $K$, $K$ acts on
$\frk$ by the adjoint representation and $ \overline{ix}$ denotes the $K$-conjugacy class of ${ix}$. Define the {\it eigencone}
(for any $s\geq 1$):
$$\EV(s,K)=\{(\bar{k}_1,\dots,\bar{k}_s)\in(\frk/K)^s: \exists k_j\in
\bar{k}_j
 \,\text{with}\,\sum_{j=1}^s
k_j=0\}.$$

For an algebraic group homomorphism $G\to G'$ 
which takes  $K\to
K'$, there is an induced map $\EV(s,K)\to \EV(s,K')$. Therefore,
$\EV(s,K)$ is functorial. In the case of $\Sp(2n)\hookrightarrow\SL(2n)$
(similarly $\SO(2n+1)\hookrightarrow \SL(2n+1)$), we have stronger properties
described below:

 Let $\frh^{\Sp(2n)}$ be the Cartan subalgebra for $\Sp(2n)$ and
similarly for $SO(2n+1), \SL(2n)$ and $\SL(2n+1)$. There are natural
(linear) embeddings (see Sections ~\ref{notate1}
and ~\ref{notate2}):
$$\frh_{+}^{\Sp(2n)}
\hookrightarrow \frh_{+}^{\SL(2n)},\,\,\,\text{and}\,\,
\frh_{+}^{\SO(2n+1)}
\hookrightarrow \frh_{+}^{\SL(2n+1)}.$$ Identify $\frk/K$ with $ \frh_{+}$ under $C$ as above. 
We state our first main theorem
(cf. Theorem ~\ref{woodrow1}):
\begin{theorem}\label{woodrow'1}
\begin{enumerate}
\item[(a)]  For $h_1,\dots,h_s\in \frh_{+}^{\Sp(2n)}$, $$(h_1,\dots,h_s)\in
\EV(s,\Sp(2n))\Leftrightarrow (h_1,\dots,h_s)\in
\EV(s,\SU(2n)).$$
\item[(b)] For $h_1,\dots,h_s\in \frh_{+}^{\SO(2n+1)}$,
$$(h_1,\dots,h_s)\in
\EV(s,\SO(2n+1))
\Leftrightarrow(h_1,\dots,h_s)\in
\EV(s,\SU(2n+1)).$$
\end{enumerate}
\end{theorem}
 Our proof of  this theorem is a
consequence of a surprising result in intersection theory described
below.
\subsection{Intersection theory}
Let $V$ be a $2n$-dimensional complex vector space equipped with a
nondegenerate symplectic form $\langle \,,\,\rangle$.
Recall that if $A=\{a_1<a_2<\cdots<a_m\}$ is a subset of $[2n]
:=\{1, \dots, 2n\}$ of cardinality $m$ and
$F_{\bull}$ a complete flag on $\Bbb{C}^{2n}$, then, by definition,
the corresponding Schubert cell $\OMM_{A}(F_{\bull}):=\{X\in
\Gr(m,\Bbb{C}^{2n}):\dim X\cap F_{u}=\ell$ for $a_{\ell}\leq u
<a_{\ell+1},\ell=1,\dots,m\}$, where 
$\Gr(m,\Bbb{C}^{2n})$ is the ordinary Grassmannian of
$m$-dimensional subspaces of $\Bbb{C}^{2n}$.

 Let $A^1,\dots,A^s$ be subsets of
$[2n]$ each of cardinality $m$. Let $E^1_{\bull}, \dots, 
E^s_{\bull}$ be
complete {\it isotropic} flags on $V$ in general
position. The following second main theorem (cf. Theorem ~\ref{wilson1}) is a key technical result that
underlies the proof of the first main theorem and, in fact, is central to our
 paper.

\begin{theorem}\label{wilson1}
The intersection $\cap_{j=1}^s \OMM_{A^j}(E^j_{\bull})$ of
subvarieties of $\Gr(m,V)$ is
 proper (possibly empty) for complete isotropic flags $E^j_{\bull}$ on $V$ in general
position.
\end{theorem}

 We have the same result  for $\SO(2n+1)$ (cf. Theorem ~\ref{newwilson2}).
Let $V'$ be a vector space of dimension  $2n+1$   equipped with a
nondegenerate symmetric bilinear form. Let $A^1,\dots,A^s$ be
subsets of $[2n+1]$ each of cardinality $m$. Let
$E^1_{\bull}, \dots, E^s_{\bull}$ be isotropic flags on $V'$ in general
position.
\begin{theorem}\label{wilson2}
The intersection $\cap_{j=1}^s \OMM_{A^j}(E^j_{\bull})$  of
subvarieties of $\Gr(m,V')$ is
 proper.
\end{theorem}

We remark that the above intersection result is false for $\SO(2n)$, since
 Corollary ~\ref{frasier} can easily seen to be violated in this case even for $m=1$.

\subsection{The saturation problem}\label{forth} Any dominant weight
 $\lambda$ of
$\SL(2n)$ restricts to a dominant weight $\lambda_C$ of the
symplectic group $\Sp(2n)$.  Similarly, any dominant weight
 $\lambda$ of
$\SL(2n+1)$ restricts to a dominant weight $\lambda_B$ of the orthogonal
group $\SO(2n+1)$. The following theorem  (cf. Theorem ~\ref{clef'1})
is proved geometrically by
the method of ``theta sections''. We obtain invariants in tensor
products by constructing divisors in the products of isotropic flag
varieties (``the theta divisor''). This technique originates in our
context in ~\cite{invariant}. In addition to methods from
~\cite{invariant}, ~\cite{gh}, we use some ideas of Schofield
~\cite{sch} in a crucial manner.
\begin{theorem}\label{clef1}
Let $V_{\lambda^1},\dots,V_{\lambda^s}$ be irreducible representations of
$\SL(2n)$ (with highest weights $\lambda^1, \dots, \lambda^s$ respectively)
 such that their tensor product has a nonzero $\SL(2n)$-invariant. Then, the tensor
 product of the  representations of $\Sp(2n)$ with 
highest weights  $\lambda_C^1,\dots,\lambda_C^s$ has a nonzero
$\Sp(2n)$-invariant, where $\lambda^j_C$ is the restriction of $\lambda^j$ to the Cartan subalgebra of $\Sp(2n)$.
\end{theorem}
A similar property holds for the odd orthogonal group $\SO(2n+1)$.
\begin{theorem}\label{clef2}
Let $V_{\lambda^1},\dots, V_{\lambda^s}$ be irreducible representations of $\SL(2n+1)$
such that their tensor product has a nonzero $\SL(2n+1)$-invariant.
Then, the tensor product of the  representations of  $\SO(2n+1)$ with
highest weights  $\lambda_B^1,\dots,\lambda_B^s$  has a nonzero
$\SO(2n+1)$-invariant,  where $\lambda^j_B$ is the restriction of $\lambda^j$ to the Cartan subalgebra of $\SO(2n+1)$.
\end{theorem}

Using the saturation theorem of Knutson-Tao [KT], together with our Theorems
~\ref{clef1} and  ~\ref{woodrow'1}, we obtain the
following improvement of the general Kapovich-Millson saturation theorem in
the case of 
$\Sp(2n)$ (cf. Theorem ~\ref{saturation}).
\begin{theorem}\label{saturation}
Given dominant integral weights  $\mu^1,\dots,\mu^s$ of $\Sp(2n)$, the following
are equivalent:
\begin{enumerate}
\item For some $N\geq 1$, the tensor product of representations of $\Sp(2n)$
with highest weights
$N\mu^1, \dots,N\mu^s$ has a nonzero $\Sp(2n)$-invariant.
\item The
tensor product of representations with highest weights  $2\mu^1,
\dots,2\mu^s$ has a nonzero $\Sp(2n)$-invariant.
\end{enumerate}
\end{theorem}
By carrying out a similar analysis for the odd orthogonal groups, we
obtain the following (cf. Theorem ~\ref{saturation2}).
\begin{theorem}\label{saturation2}
Given dominant integral weights $\nu^1,\dots,\nu^s$ of $\SO(2n+1)$, the
following are equivalent:
\begin{enumerate}
\item For some $N\geq 1$, the tensor product of representations with highest weights
$N\nu^1, \dots,N\nu^s$ has a nonzero $\SO(2n+1)$-invariant.
\item The
tensor product of representations with highest weights $2\nu^1,
\dots,2\nu^s$ has a nonzero $\SO(2n+1)$-invariant.
\end{enumerate}
\end{theorem}
\subsection{Horn's problem for symplectic and odd orthogonal
groups}\label{jjj} There are two senses in which the term ``Horn's
problem'' is generally  applied:
\begin{enumerate}
\item[(A)] Determination of a system of inequalities for the
eigencone $\Gamma(s,K)$, which  is ``cohomology free'' (but perhaps
recursive).
\item[(B)]
A system of inequalities characterizing nonvanishing structure
coefficients in the cohomology  $H^*(G/P)$ with the standard cup product and
 also $(H^*(G/P),\odot_0)$,  where recall that
$\odot_0$  is a
deformation of the  cup product in the cohomology 
of the flag variety $G/P$ introduced in ~\cite{BK}. In this paper, we only 
consider the question in the case of  $(H^*(G/P),\odot_0)$. 
 The system of inequalities can 
reasonably be expected
 to be parameterized by  the  cohomology $(H^*(L/Q),\odot_0)$, where $L$ is
 a Levi subgroup of $G$ and $Q\subset L$ is a parabolic subgroup. 
It was suggested in [BK] that nonvanishing structure
 coefficients in the deformed cohomology can be characterized
  recursively. Recently,
 Richmond  ~\cite{ed} showed that the cohomology $(H^*(SL_n/P),\odot_0)$ of partial flag varieties
 is strongly recursive (not just nonvanishing of structure coefficients, but also the actual structure coefficients).
 \end{enumerate}
 We  consider (A) and (B) for the groups $G=\Sp(2n)$  and $\SO(2n+1)$.
  Theorems ~\ref{clef1} and ~\ref{clef2} imply (A), because the eigenvalue problem
for the special linear groups is recursive by Horn's original conjecture,
proved by the combined works of Klyachko and Knutson-Tao (see
~\cite{ful} for a survey).

It follows from ~\cite{BK} that (B) implies (A). We reduce the Horn problem
 (B)  for $(H^*(G/P),\odot_0)$ for the groups $G=\Sp(2n)$ and $\SO(2n+1)$ and maximal parabolic subgroups $P$  to the corresponding problem 
for the  Lagrangian Grassmannians  in Sections
~\ref{july4} and 9 (cf. Theorems ~\ref{challenge} and ~\ref{challenge1}). Now, the Horn problem for the  Lagrangian Grassmannians was solved 
by  Purbhoo-Sottile
~\cite{ps}.
It should be noted that the solution of (A)  obtained
through the solution for (B) is different from the one obtained from
Theorems ~\ref{clef1} and ~\ref{clef2}. 

We make a conjecture (cf. Conjecture ~\ref{conj}) generalizing many of the results in the paper for a diagram automorphism of $G$. 

{\it We establish the basic notation in Section 2 and use them through the paper
often without further explanation.}

\vskip2ex \noindent {\it Acknowledgements.} Both the  authors were
supported by the FRG grant no DMS-0554247 from NSF.
\section{Notation and Preliminaries}

Let $G$ be a connected semisimple complex algebraic group.  We fix a
Borel subgroup $B$ and a maximal torus $H\subset B$. Their Lie
algebras are denoted by the corresponding Gothic characters:
$\mathfrak g, \mathfrak b, \mathfrak h$ respectively. Let $W=N(H)/H$
be the Weyl group of $G$, $N(H)$ being the normalizer of $H$ in $G$.
Let $R^+\subset \frh^*$ be the set of positive roots (i.e., the set
of roots of $\frb$) and $\Delta=\{\alpha_1, \dots, \alpha_n\}
\subset R^+$ the set of simple roots. Let
 $\{\alpha_1^\vee, \dots, \alpha_n^\vee\}\subset \frh$ be the set of corresponding
 simple coroots and $\{s_1, \dots, s_n\}\subset W$ be the set of corresponding
 simple reflections. Let $\frh_+\subset \frh$ be the dominant chamber defined by
 \[\frh_+=\{x\in \frh: \alpha_i(x)\in \Bbb R_+ \,\forall \alpha_i\},\]
 where $\Bbb R_+$ is the set of nonnegative real numbers. For any $1\leq i\leq n$,
 let $\omega_i \in \frh^*$ denote the $i$-th fundamental weight defined by
 \[\omega_i(\alpha_j^\vee)=\delta_{i,j}.\]

 Let $B\subset P$ be a (standard) parabolic subgroup with the unique Levi subgroup $L$
 containing $H$. We denote by $W^P$ the set of minimal-length coset representatives
 in the $W_P$-cosets $W/W_P$, where $W_P$ is the Weyl group of $P$ (which is, by
 definition, the Weyl group of the Levi subgroup $L$). For any $w\in W^P$, we have the Bruhat cell
 \[\Lambda_w^P:=BwP/P \subset G/P.\]
 This is a locally closed subset of the flag variety $G/P$ isomorphic to the affine space
 $\Bbb A^{\ell(w)}, \ell(w)$ being the length of $w$. Its closure is denoted by
 $\bar{\Lambda}^P_w$, which is an irreducible projective variety (of dimension $\ell(w)$).
We denote by
\[[\bar{\Lambda}^P_w]\in H^{2(\dim_\Bbb C(G/P)-\ell (w))}\,(G/P)\]
the cycle class of the subvariety  $\bar{\Lambda}^P_w$, where
$H^*(G/P)$ is the singular cohomology of $G/P$ with integral coefficients. Then, by
the Bruhat decomposition, $\{[\bar{\Lambda}^P_w]\}_{w\in W^P}$ is an integral basis of $H^*(G/P)$.
Define the basis $\{x_1, \dots, x_n\}$ of $\frh$ dual to the basis
$\{\alpha_1, \dots, \alpha_n\} $ of $\frh^*$, i.e.,
\[\alpha_j(x_i)=\delta_{i,j}.\]

Let $X(H)_+$ be the set of dominant characters of $H$. For any $\lambda
\in X(H)_+$, let $V_\lambda$ be the finite dimensional irreducible $G$-module
with highest weight $\lambda$. This sets up a bijective correspondence between
$X(H)_+$ and the set of isomorphism classes of  finite dimensional irreducible $G$-modules. Taking the derivative, we get an embedding $X(H)_+
\hookrightarrow D$, where $$D:=\{\lambda \in \frh^*: \lambda (\alpha_i^\vee)
\in
\Bbb R_+ \,\forall \alpha_i^\vee \}$$ is the set of dominant weights. If $G$ is simply-connected, $X(H)_+$ can be identified with
$$D_\Bbb Z:=\{\lambda \in \frh^*: \lambda (\alpha_i^\vee)
\in
\Bbb Z_+ \,\forall \alpha_i^\vee \},$$
where $\Bbb Z_+$ is the set of nonnegative integers.

We now give more specific details about the groups $\SL(n+1), \Sp(2n)$ and $\SO(2n+1)$ below
(see, e.g., [BL, Chapter 3]), since they will be of special interest to us in the paper.

\subsection{Special Linear Group $\SL(n+1)$.}
In this case we take  $B$ to be the (standard)
Borel subgroup consisting of upper triangular matrices of determinant $1$
and $H$ to be the subgroup
 consisting of diagonal matrices (of determinant $1$). Then,
 \[\frh=\{{\bf t}=\text{diag}(t_1, \dots , t_{n+1}): \sum t_i=0\}, \]
 and
  \[\frh_+=\{{\bf t}\in \frh: t_i\in \Bbb R\,\text{and}\, t_1\geq \cdots  \geq t_{n+1}\}.\]
  For any $1\leq i\leq n$,
  \[\alpha_i({\bf t})=t_i-t_{i+1}; \alpha_i^\vee=\text{diag}(0, \dots, 0,1,-1,0, \dots,
  0); \omega_i({\bf t})=t_1+\cdots t_i,\]
  where $1$ is placed in the $i$-th place.

  The Weyl group $W$ can be identified with the symmetric group $S_{n+1}$ which acts
  via the permutation of the coordinates of ${\bf t}$. Let $\{r_1, \dots, r_n\} \subset S_{n+1}$
  be the (simple)
  reflections corresponding to the simple roots $\{\alpha_1, \dots, \alpha_n\}$
  respectively. Then,
  \[r_i=(i,i+1).\]
  For any $1\leq m\leq n$, let $P_m\supset B$ be the (standard) maximal parabolic
  subgroup of $\SL(n+1)$ such that its unique Levi subgroup $L_m$ containing $H$
  has for its simple roots $\{\alpha_1, \dots, \hat{\alpha}_m, \dots, \alpha_n\}$.
  Then, $\SL(n+1)/P_m$ can be identified with the Grassmannian $\Gr(m, n+1)
  =\Gr(m, \Bbb C^{n+1})$ of $m$-dimensional subspaces of $\Bbb C^{n+1}$. Moreover, the set
  of minimal coset representatives $W^{P_m}$ of $W/W_{P_m}$ can be identified
with the set of $m$-tuples
\[S(m,n+1)=\{A:=1\leq a_1 < \cdots < a_m \leq n+1 \}.\]
Any such $m$-tuple $A$ represents the permutation
\[v_A=(a_1,\dots, a_m,a_{m+1}, \dots, a_{n+1}),\]
where $\{a_{m+1}< \cdots < a_{n+1}\}=[n+1]\setminus \{a_1, \dots, a_m\}$
and
\[[n+1]:=\{1, \dots, n+1\}.\]

For a complete flag
 $E_{\bull}:0=E_0\subsetneq E_1\subsetneq
\cdots\subsetneq E_{n+1}= \Bbb C^{n+1}$, and $A \in S(m,n+1)$, define the
corresponding
{\it shifted Schubert cell} inside $\Gr(m,n+1)$:
$$\OMM_{A}(E_{\bull})=\{M\in \Gr(m,n+1): \,\text{for any}\, 0\leq \ell
\leq m \, \text{and any}\, a_\ell\leq b < a_{\ell +1}, \dim M\cap
E_{b}=\ell\},$$ where we set $a_0=0$ and $a_{m+1}=n+1.$ Then,
$\OMM_{A}(E_{\bull})=g(E_{\bull}) \Lambda_{v_A}^{P_m}$, where
$g(E_{\bull})$ is an element of $\SL(n+1)$ which takes the standard
flag $E_{\bull}^o$ to the flag $E_{\bull}$. (Observe that
$g(E_{\bull})$ is determined up to the right multiplication by an
element of $B$.) Its closure in $\Gr(m,n+1)$ is denoted by
$\OMC_{A}(E_{\bull})$ and its cycle class in
$H^*(\Gr(m,n+1))$ by $[{\OMC}_{A}]$. (Observe that the cohomology
class $[{\OMC}_{A}]$ does not depend upon the choice of
$E_{\bull}$.) For the standard flag $E_{\bull}=E^o_{\bull}$, we thus
have
 $\OMM_{A}(E_{\bull})=\Lambda_{v_A}^{P_m}$.

\begin{remark}
{Note that, in the literature,  it is
 more common to denote the Schubert cell by
$\Omega^o_A(E_{\bull})$ and its closure by $\Omega_A(E_{\bull})$.
For notational uniformity we have denoted the Schubert cell by
$\Omega_A(E_{\bull})$, and its closure by
$\bar{\Omega}_A(E_{\bull})$.}
\end{remark}
\subsection{Symplectic Group $\Sp(2n)$.}\label{notate1} Let $V=\Bbb{C}^{2n}$ be equipped with the
nondegenerate symplectic form $\langle \,,\,\rangle$ so that its matrix
$\bigl(\langle e_i,e_j\rangle\bigr)_{1\leq i,j \leq 2n}$ in the
standard basis $\{e_1,\dots, e_{2n}\}$ is given by
\begin{equation*}
E=\left(\begin{array}{cc}
0&J\\
-J&0
\end{array}\right),
\end{equation*}
where $J$ is the anti-diagonal matrix $(1,\dots,1)$ of size $n$. Let 
$$\Sp(2n):=\{g\in \SL(2n):
g \,\text{leaves  the form}\, \langle \,,\,\rangle \,\text{invariant}\}$$ be the associated
symplectic group.  Clearly, $\Sp(2n)$ can be realized
as the fixed point subgroup $G^\sigma$ under the involution $\sigma:G\to G$
defined by $\sigma(A)=E(A^t)^{-1}E^{-1}$, where $G=\SL(2n)$. The involution $\sigma$
keeps both of $B$ and $H$ stable, where $B$ and $H$  are as in the $\SL(2n)$ case. Moreover,
$B^\sigma$ (respectively, $H^\sigma$) is a Borel subgroup (respectively, a maximal torus)
of $\Sp(2n)$. We denote $B^\sigma, H^\sigma$ by $B^C=B^{C_n},H^C=H^{C_n}$
respectively and (when confusion is likely) $B,H$ by $B^{A_{2n-1}}, H^{A_{2n-1}}$
respectively (for $\SL(2n)$). Then, the Lie algebra of $H^C$ (the Cartan subalgebra
$\frh^C$)
\[\frh^C=\{\text{diag}(t_1, \dots, t_n, -t_n,\dots, -t_1):t_i\in \Bbb C\}.\]
Let $\Delta^C=\{\beta_1, \dots, \beta_n\}$ be the set of simple roots. Then, for any
$1\leq i\leq n$, $\beta_i={\alpha_i}_{\vert\frh^C},$ where $\{\alpha_1, \dots, \alpha_{2n-1}\}$
are the simple roots of $\SL(2n)$. The corresponding (simple) coroots
$\{\beta_1^\vee, \dots, \beta_n^\vee\}$ are given by
\[\beta^\vee_i=\alpha_i^\vee+ \alpha_{2n-i}^\vee, \,\,\,\text{for}\, 1\leq i <n\]
and
\[\beta_n^\vee=\alpha^\vee_n.\]
Thus,
\[\frh^C_+=\{\text{diag}(t_1, \dots, t_n, -t_n,\dots, -t_1):\text{each $t_i$
is   real and}\, t_1\geq \cdots \geq t_n\geq 0\}.\]
Moreover, $\frh^{A_{2n-1}}_+$ is $\sigma$-stable and
\[ \bigl(\frh^{A_{2n-1}}_+\bigr)^\sigma = \frh^C_+.\]
Let $\{s_1, \dots, s_n\}$ be the (simple) reflections in the Weyl group
$W^C=W^{C_n}$ of $\Sp(2n)$ corresponding to the simple roots $\{\beta_1, \dots, \beta_n\}$
respectively.
Since $H^{A_{2n-1}}$ is $\sigma$-stable, there is an induced action of $\sigma$ on the Weyl group
$S_{2n}$ of $\SL(2n)$.
The Weyl group $W^C$  can be identified with the subgroup of $S_{2n}$
consisting of $\sigma$-invariants:
\[ \{(a_1,\dots, a_{2n})\in S_{2n}: a_{2n+1-i}=2n+1-a_i \, \forall 1\leq i\leq 2n\}.\]
In particular, $w=(a_1,\dots, a_{2n})\in W^C$ is determined from
$(a_1,\dots, a_{n})$.

Under the inclusion  $W^C\subset S_{2n}$, we have
 \begin{align} \label{0}
s_i&=r_ir_{2n-i}, \,\,\,\text{if}\,\,\, 1\leq i\leq n-1\notag\\
 &=r_n, \,\,\,\text{if}\,\,\, i=n.
 \end{align}
Moreover, for any $u,v\in W^C$ such that  $\ell^C(uv)= \ell^C(u)+ \ell^C(v)$,
we have
 \begin{equation} \label{-1}
 \ell^{A_{2n-1}}(uv)= \ell^{A_{2n-1}}(u)+ \ell^{A_{2n-1}}(v),
\end{equation}
where $\ell^C(w)$ denotes the length of $w$ as an element of the
Weyl group $W^C$ of $\Sp(2n)$ and similarly for $\ell^{A_{2n-1}}$.

For  $1\leq r \leq n$, we let $\IG(r,2n)=\IG(r,V)$ to be
the set of $r$-dimensional isotropic
subspaces of $V$ with respect to the form $\langle\,,\,\rangle$, i.e.,
$$\IG(r,2n):=\{M\in \Gr(r,2n): \langle v,v'\rangle=0,\ \forall\,  v,v'\in M\}.$$
 Then, it is the quotient $\Sp(2n)/P_r^C$ of $\Sp(2n)$ by
 the standard maximal parabolic subgroup
$P_r^C$  with $\Delta^C\setminus \{\beta_r\}$ as the set of simple roots of its Levi component
$L_r^C$. (Again we take $L_r^C$ to be the unique Levi subgroup of $P_r^C$
 containing $H^C$.) It can be easily seen that the set
$W^C_r$ of minimal-length coset representatives of $W^C/W_{P_r^C}$ is identified with the set
\[\FS(r,2n)=\{{I}:= 1\leq i_1< \cdots <i_r\leq 2n \,\,{\text and}\,\, I\cap \bar{I}=
\emptyset\},\]
where
\begin{equation}\label{defibari}
\bar{I}:=\{2n+1-i_1, \dots, 2n+1-i_r\}.
\end{equation}
 Any such $I$ represents the
permutation $w_I=(i_1,\dots, i_{n}) \in W^C$ by taking
$\{i_{r+1}<\cdots <i_n\}=[n]\setminus(I\sqcup \bar{I}).$

\begin{defi}\label{woodyallen}
A complete flag $$E_{\bull}:0=E_0\subsetneq E_1\subsetneq
\cdots\subsetneq E_{2n}= V$$
 is called an {\rm isotropic flag} if $E_{a}^{\perp}=E_{2n-a}$, for
 $a=1,\dots,2n$. (In particular,  $E_n$ is a maximal isotropic
subspace of $V$.)

For an isotropic flag $E_{\bull}$ as above, there exists an element
$k(E_{\bull})\in \Sp(2n)$  which takes the standard flag
$E_{\bull}^o$ to the flag $E_{\bull}$. (Observe that $k(E_{\bull})$ is determined
up to the right multiplication by an element of $B^C$.)
\end{defi}

For any $I\in \FS(r,2n)$ and any isotropic flag $E_{\bull}$, we have the
corresponding {\it shifted Schubert cell} inside $\IG(r,V)$:
$$\LAM_{I}(E_{\bull})=\{M\in \IG(r,V): \,\text{for any}\,
0\leq \ell \leq r \,
\text{and any}\, i_\ell\leq a < i_{\ell +1}, \dim M\cap E_{a}=\ell\},$$
where we set $i_0=0$ and $i_{r+1}=2n$. Clearly, set
theoretically,
\begin{equation}\label{2}\LAM_I(E_{\bull})=\OMM_I(E_{\bull})\cap
\IG(r,V);
\end{equation}
this is also a scheme theoretic equality (cf. Proposition ~\ref{morning}
(4)).
 Moreover,  $\LAM_I(E_{\bull})=k(E_{\bull})\Lambda_{w_I}^{P_r^C}$.
 Denote
the closure of $\LAM_I(E_{\bull})$ inside $\IG(r,V)$  by
${\bar{\LAM}}_I(E_{\bull})$ and its cycle class  in
$H^*(\IG(r,V))$ (which does not depend upon the choice of the
isotropic flag $E_{\bull}$) by $[\bar{\LAM}_{I}]$. For the standard
flag $E_{\bull}=E^o_{\bull}$, we have
 $\LAM_{I}(E_{\bull})=\Lambda^{P_r^C}_{w_I}$.

\subsection{Special Orthogonal Group $\SO(2n+1)$.}\label{notate2} Let $V'=\Bbb{C}^{2n+1}$ be equipped with the
nondegenerate symmetric form $\langle \,,\,\rangle$ so that its matrix $E=
\bigl(\langle e_i,e_j\rangle\bigr)_{1\leq i,j \leq 2n+1}$ (in the standard basis
 $\{e_1,\dots, e_{2n+1}\}$) is  the $(2n+1)\times (2n+1)$
antidiagonal matrix with $1'$s all along the antidiagonal except at the $(n+1, n+1)$-th
place where the entry is $2$. Note that the associated quadratic form on $V'$ is given by
\[Q(\sum t_ie_i)= t_{n+1}^2+\sum_{i=1}^n\,t_it_{2n+2-i}.\]
 Let $$\SO(2n+1):=\{g\in \SL(2n+1):
g \,\text{leaves  the quadratic  form $Q$ invariant}\}$$ be the associated
special orthogonal group.  Clearly, $\SO(2n+1)$ can be realized
as the fixed point subgroup $G^\theta$ under the involution $\theta:G\to G$
defined by $\theta(A)=E^{-1}(A^t)^{-1}E$, where $G=\SL(2n+1)$. The involution $\theta$
keeps both of $B$ and $H$ stable. Moreover,
$B^\theta$ (respectively, $H^\theta$) is a Borel subgroup (respectively, a maximal torus)
of $\SO(2n+1)$. We denote $B^\theta, H^\theta$ by $B^B=B^{B_n},H^B=H^{B_n}$
respectively. Then, the Lie algebra of $H^B$ (the Cartan subalgebra
$\frh^B$)
\[\frh^B=\{\text{diag}(t_1, \dots, t_n, 0,-t_n,\dots, -t_1):t_i\in \Bbb C\}.\]
This allows us to identify $\frh^C$ with $\frh^B$ under the map
\[\text{diag}(t_1, \dots, t_n, -t_n,\dots, -t_1)\mapsto \text{diag}(t_1, \dots,
t_n, 0,-t_n,\dots, -t_1).\]
Let $\Delta^B=\{\delta_1, \dots, \delta_n\}$ be the set of simple roots. Then, for any
$1\leq i\leq n$, $\delta_i={\alpha_i}_{\vert\frh^B},$ where $\{\alpha_1, \dots, \alpha_{2n}\}$
are the simple roots of $\SL(2n+1)$. The corresponding (simple) coroots
$\{\delta_1^\vee, \dots, \delta_n^\vee\}$ are given by
\[\delta^\vee_i=\alpha_i^\vee+ \alpha_{2n+1-i}^\vee, \,\,\,\text{for}\, 1\leq i <n\]
and
\[\delta_n^\vee=2(\alpha^\vee_n+\alpha^\vee_{n+1}).\]
Thus, under the above identification,
\[\frh^B_+=\frh^C_+.\]
Moreover, $\frh^{A_{2n}}_+$ is $\theta$-stable and
\[ \bigl(\frh^{A_{2n}}_+\bigr)^\theta = \frh^B_+.\]
Let $\{s'_1, \dots, s'_n\}$ be the (simple) reflections in the Weyl group
$W^B=W^{B_n}$ of $\SO(2n+1)$ corresponding to the simple roots $\{\delta_1, \dots, \delta_n\}$
respectively.
Since $H^{A_{2n}}$ is $\theta$-stable, there is an induced action of $\theta$ on the Weyl group
$S_{2n+1}$ of $\SL(2n+1)$.
The Weyl group $W^B$  can be identified with the subgroup of $S_{2n+1}$
consisting of $\theta$-invariants:
\[ \{(a_1,\dots, a_{2n+1})\in S_{2n+1}: a_{2n+2-i}=2n+2-a_i \, \forall 1\leq i\leq 2n+1\}.\]
In particular, $w=(a_1,\dots, a_{2n+1})\in W^B$ is determined from
$(a_1,\dots, a_{n})$. (Observe that $a_{n+1}=n+1$.) We can identify
the Weyl groups $W^C\simeq W^B$ under the map $(a_1,\dots, a_{2n})\mapsto (a_1,
\dots,  a_n, n+1, a_{n+1}+1, \dots, a_{2n}+1).$

Under the inclusion  $W^B\subset S_{2n+1}$, we have
 \begin{align}
 s'_i&=r_ir_{2n+1-i}, \,\,\,\text{if}\,\,\, 1\leq i\leq n-1\notag\\
 &=r_nr_{n+1}r_n, \,\,\,\text{if}\,\,\, i=n.
 \end{align}

For  $1\leq r \leq n$, we let $\OG(r,2n+1)=\OG(r,V')$ to be
the set of $r$-dimensional isotropic
subspaces of $V'$ with respect to the quadratic form $Q$, i.e.,
$$\OG(r,2n+1):=\{M\in \Gr(r,V'): Q(v)=0,\ \forall\,  v\in M\}.$$
 Then, it is the quotient $\SO(2n+1)/P_r^B$ of $\SO(2n+1)$ by
 the standard maximal parabolic subgroup
$P_r^B$  with $\Delta^B\setminus \{\delta_r\}$ as the set of simple roots of its Levi component
$L_r^B$. (Again we take $L_r^B$ to be the unique Levi subgroup of $P_r^B$
 containing $H^B$.) It can be easily seen that the set
$W^B_r$ of minimal-length coset representatives of $W^B/W_{P_r^B}$ is identified with the set
\[\FS'(r,2n+1)=\{{J}:= 1\leq j_1< \cdots <j_r\leq 2n+1 ,  j_p\neq n+1 \,
\text{for\, any}\, p \,\,{\text and}\,\, J\cap \bar{J}'=
\emptyset\},\]
where
$$\bar{J}':=\{2n+2-j_1, \dots, 2n+2-j_r\}.$$
Any such $J$ represents the permutation $w'_J=(j_1,\dots,
j_{n})
\in W^B$ by taking
$\{j_{r+1}<\cdots <j_n\}=[n]\setminus(J\sqcup \bar{J}').$

Similar to the Definition \ref{woodyallen} of isotropic flags on $V$, we have
 the notion of isotropic flags on $V'$.
Then, for an isotropic flag $E'_{\bull}$, there exists an element
$k(E'_{\bull})\in \SO(2n+1)$  which takes the standard flag
${E'}_{\bull}^o$ to the flag $E'_{\bull}$. (Observe that $k(E'_{\bull})$ is determined
up to the right multiplication by an element of $B^B$.)

For any $J\in \FS'(r,2n+1)$ and any isotropic flag $E'_{\bull}$, we have the
corresponding {\it shifted Schubert cell} inside $\OG(r,V')$:
$$\Psi_{J}(E'_{\bull})=\{M\in \OG(r,V'): \,\text{for any}\,
0\leq \ell \leq r \,
\text{and any}\, j_\ell\leq a < j_{\ell +1}, \dim M\cap E'_{a}=\ell\},$$
where we set $j_0=0$ and $j_{r+1}=2n+1$. Clearly, set
theoretically,
\begin{equation}\Psi_J(E'_{\bull})=\OMM_J(E'_{\bull})\cap
\OG(r,V');
\end{equation}
this is also a scheme theoretic equality.
 Moreover,  $\Psi_J(E'_{\bull})=k(E'_{\bull})\Lambda_{w'_J}^{P_r^B}$.
 Denote
the closure of $\Psi_J(E'_{\bull})$ inside $\OG(r,V')$  by
$\bar{\Psi}_J(E'_{\bull})$
and its cycle class  in $H^*(\OG(r,V'))$
 (which does not depend upon the choice of the isotropic flag
$E'_{\bull}$) by $[\bar{\Psi}_{J}]$.
For
the standard flag $E'_{\bull}=E^o_{\bull}$, we have
 $\Psi_{J}(E'_{\bull})=\Lambda^{P_r^B}_{w'_J}$.

\section{Isotropic flags and proper intersection of Schubert cells in
$\Gr(m,V)$}\label{three3}

Fix a positive integer $s$. Let $V=\Bbb C^{2n}$ be equipped with the nondegenerate
 symplectic form $\langle\,,\,\rangle$ as in Section 2, and let $1\leq m\leq n$ be
 a positive integer. Let $A^1,\dots,A^s\in S(m,2n)$. The following theorem is a key technical result that
underlies the proof of our theorem on the comparison of eigencone for $\Sp(2n)$ with that
of $\SL(2n)$.

\begin{theorem}\label{wilson1} Let $E^1_{\bull},\dots,E^s_{\bull}$ be
isotropic flags on $V$  in general position. Then, the  intersection
of subvarieties $\cap_{j=1}^s \OMM_{A^j}(E^j_{\bull})$ inside
$\Gr(m,V)$ is
 proper (possibly empty).
\end{theorem}
As an immediate consequence of the above theorem, we get the following:

\begin{corollary}\label{frasier} Let $1\leq m\leq n$ and let
 $I^1,\dots, I^s\in \FS(m,2n)$  be such that
$$\prod_{j=1}^s[\bar{\LAM}_{{I^j}}]\neq 0\in H^{*}(\IG(m,2n)).$$
 Then,  $\prod_{j=1}^s[\OMC_{{I^j}}]\neq 0\in H^{*}(\Gr(m,2n)).$
\end{corollary}
\begin{proof} Observe that by [F$_2$, Proposition 7.1 and Section 12.2] and [B$_1$, Proposition 1.1],
\begin{equation} \label{3}\prod_{j=1}^s[\bar{\LAM}_{{I^j}}]\neq 0 \,\,\text{if
\,and \,only\, if}\,\, \cap_{j=1}^s
\,{\LAM}_{{I^j}}(E^j_{\bull})\neq \emptyset
\end{equation}
for isotropic flags $\{E^j_{\bull}\}$ such that the above intersection is
 proper. Thus, by
assumption, $\cap_{j=1}^s \,{\LAM}_{I^j}(E^j_{\bull})\neq \emptyset$
for such flags $\{E^j_{\bull}\}$. By the above theorem and Equation
~\eqref{2}, we conclude that $\cap_{j=1}^s
\,{\OMM}_{{I^j}}(E^j_{\bull})\neq \emptyset$. From this and using
Equation \eqref{3} for $\Gr(m, V)$, the corollary follows.
\end{proof}


Before we can prove the above theorem, we need the following preliminary work.

\vskip2ex

Given subsets
$I$ and $K$ of $[2n]$,  we denote by $|I > K|$ the number of
pairs $(i,k)$ with $i\in I$, $k\in K$ and $i> k$.  (We set
 $|A>\emptyset|=0$.) For $K$ a singleton $\{k\}$, we abbreviate
 $|I>K|$ by $|I>k|$. Let us take an integer $1\leq r\leq n$.
For an $I\in \FS(r, 2n)$, we define
$$\tilde{I}= [2n]\setminus (I\sqcup \bar{I}).$$
(See Equation ~\eqref{defibari} for the definition of $\bar{I}$.) We
also set (for any $I\in \FS(r,2n)$)
$$\sym^2(I)= \frac{1}{2}(|I>\bar{I}|+\mu(w_I)),$$
and
$$\wedge^2(I)=|I>\bar{I}|- \sym^2(I),$$
where $\mu(w_I)$ represents the number of times the simple reflection $s_n$
appears in any
reduced decomposition of $w_I$.

With this notation, we have the
following proposition.

\begin{proposition}\label{morn2} For any $I\in \FS(r,2n)$ and any isotropic flag $E_{\bull}$,
$$\dim \LAM_I(E_{\bull})=|I> \tilde{I}| + \sym^2(I).$$
\end{proposition}

\begin{proof} Under the canonical inclusion $W^C\hookrightarrow S_{2n}$,
$$w_I\mapsto {\hat{w}}_I=(i_1, \dots, i_r, j_1, \dots, j_{n-r},
2n+1-j_{n-r}, \dots, 2n+1-j_1,2n+1-i_r, \dots, 2n+1-i_1),$$ where
$\{j_1<\cdots <j_{n-r}\}:=[n]\setminus (I\sqcup \bar{I})$. Now, by
Equations ~\eqref{0}, ~\eqref{-1} and  [F$_1$, $\S$10.2],
\begin{align*}
\ell^C(w_I)&=\frac{1}{2}\bigl(\ell^{A_{2n-1}} ({\hat{w}}_I)+\mu(w_I)\bigr)\\
&=\frac{1}{2}\bigl(|I>\tilde{I}|+|I>\bar{I}|+|\tilde{I}>\bar{I}|+
\mu(w_I)\bigr)\\
&=\frac{1}{2}\bigl(|I>\tilde{I}|+|I>\bar{I}|+|I>\tilde{I}|+
\mu(w_I)\bigr)\\
&=|I>\tilde{I}|+ \sym^2(I).
\end{align*}
This proves the proposition.
\end{proof}

\subsection{Tangent space of isotropic Grassmannians}\label{two2}
 We calculate now the tangent space $T(X)_M$ to
$X=\IG(r,V)$ at a point $M$. Because of the
natural embedding $\IG(r,V)\subseteq \Gr(r,V)$,  we have $T(X)_M\subseteq
T\Gr(r,V)_M =
\home(M,V/M)$.

Clearly, $M^{\perp}$ is a $2n-r$ dimensional subspace of $V$ that
contains $M$ and there is a canonical isomorphism $V/M^{\perp}\simeq
M^*$ (induced from the symplectic form). Hence, we have an exact sequence
$$0\to \home(M,M^{\perp}/M)\to \home(M,V/M)\leto{\phi} \home
(M,V/M^{\perp})=\home(M,M^*)\to 0,$$
obtained from the inclusions $M\subset M^{\perp} \subset V$. It is clear that $M^{\perp}/M$
 is a $2n-2r$ dimensional space that possesses a nondegenerate
 symplectic form. Let $P_M\subset\Sp(2n)$ be the stabilizer of $M$ and $\sym^2 M^*$ the space
 of symmetric bilinear forms on $M$.
\begin{lemma}\label{gym}
$T(X)_M=\phi^{-1}(\Sym^2 M^*)$ and hence there is an exact sequence of
$P_M$-modules.
\begin{equation}\label{fun}
0\to \home(M,M^{\perp}/M)\leto{\xi} T(X)_M\leto{\phi} \Sym^2 M^*\to 0.
\end{equation}
\end{lemma}
\begin{proof}
Let $\psi:M\to V/M$ be a linear map (viewed as a deformation of the trivial map). The deformed $M$
(obtained from $\psi$) needs to be
isotropic. So, up to the first order,  we have
$$\langle v+\ep\psi(v),v' +\ep\psi(v')\rangle=0, \forall\  v,v'\in M.$$
Hence, $$\langle v, \psi(v')\rangle + \langle \psi(v), v'\rangle=0,$$
or that
$$\langle v, \psi(v')\rangle = \langle v', \psi(v)\rangle.$$
This gives us the symmetric bilinear  form $\phi(\psi)(v,v')=\langle v,
\psi(v')\rangle$. Hence, $T(X)_M\subseteq \phi^{-1}(\Sym^2 M^*)$. But,
$$\dim T(X)_M= \dim \IG(r,V)=\frac{r}{2}(4n-3r+1)=\dim \phi^{-1}(\Sym^2 M^*),$$
 and hence   $ T(X)_M= \phi^{-1}(\Sym^2 M^*)$.
 \end{proof}

 Let $E_{\bull}$ be an isotropic
flag on $V$. This induces flags on $M$, $M^{\perp}$ and hence on
$M^{\perp}/M$.\footnote{There is exactly one way of inducing a
complete flag on $M^{\perp}/M$ from $E_{\bull}$. The elements of the
flag are $(M^{\perp}/M)\cap (E_a+M/M)  $. This can be written as
$( (E_a +M)\cap M^{\perp})/M= E_a\cap M^{\perp}/E_a\cap M$.}
\begin{lemma}\label{water}
For any isotropic flag $E_{\bull}$  on $V$, the induced flag on $M^{\perp}/M$ is
 isotropic with respect to the
nondegenerate symplectic form on $M^{\perp}/M$.
 \end{lemma}
\begin{proof}
Consider an element of the induced flag $$E_{a}\cap M^{\perp}/
E_{a}\cap M= (E_{2n-a}+M)^{\perp}/ E_{a}\cap M.$$ Its perpendicular (in
$M^{\perp}/M$) is therefore $(E_{2n-a}+M)\cap M^{\perp}/M=
E_{2n-a}\cap M^{\perp}/ E_{2n-a}\cap M$, which is again a member
of the induced flag on $M^{\perp}/M$.
\end{proof}

\subsection{Some crucial inequalities derived from nonvanishing intersection product}
For any $I\in \FS(r,2n),$ isotropic flag $E_{\bull}$ on $V$ and
$M\in \LAM_I(E_{\bull})$, we now calculate the tangent space
$T(\LAM_I(E_{\bull}))_M\subseteq T(\OMM_I(E_{\bull}))_M$ (with the
notation as in Section 2).

\begin{lemma}\label{Mper}
$M^{\perp}\in \OMM_{I\sqcup \tilde{I}}(E_{\bull}).$
\end{lemma}
\begin{proof} By definition, $I\sqcup \tilde{I}=[2n]\setminus\bar{I}$. Now,
$\dim(E_a\cap M^{\perp})= \dim(( E_{2n-a}+M)^{\perp})= 2n- (r +2n-a -\dim
 (E_{2n-a}\cap M))=a-r +\dim (E_{2n-a}\cap M). $
Hence, $E_a\cap M^{\perp}\neq E_{a-1}\cap M^{\perp}$ if and only if
$ E_{2n-a}\cap M = E_{2n+1-a}\cap M$, i.e., $2n+1-a \not\in
I\Leftrightarrow a\not\in \bar{I}$.
\end{proof}

For any $a\in [r]$, let $\lambda_a:=|i_a\geq \tilde{I}|$, where $I=
\{1\leq i_1 <\cdots <i_{r}\leq 2n\}$. Let $X=\IG(r,V)$.

 \begin{proposition}\label{morning}
 \begin{enumerate}
\item
\begin{align*}\home(M,M^{\perp}/M)\cap
T(\OMM_I(E_{\bull}))_M&=\\
 \{\gamma\in \home(M,M^{\perp}/M)&: \gamma(E(M)_a)\subset
E(M^{\perp}/M)_{\lambda_a}, \forall a\in [r]\},
\end{align*}
where $E(M)_{\bull}$ and $E(M^{\perp}/M)_{\bull}$ are the induced complete flags
on $M$ and $M^{\perp}/M$ respectively with the changed labels $\bull$ by the dimension.
The dimension of this vector space is $|I \geq \tilde{I}|$.
\item $\home(M,M^{\perp}/M)\cap
T(\LAM_I(E_{\bull}))_M=\home(M,M^{\perp}/M)\cap
T(\OMM_I(E_{\bull}))_M$.
\item
\begin{align*}
\phi(T(\LAM_I(E_{\bull}))_M)&= \phi\bigl(T(\OMM_I(E_{\bull}))_M\cap
T(X)_M\bigr)\\
&=\{\gamma\in \Sym^2 M^*: \gamma(E(M)_a, E(M)_{t_a})=0, \forall a \in [r]\},
\end{align*}
where $t_a=|\bar{I}\geq i_a|$. Moreover, the dimension of this vector
space is $\sym^2(I)$.
 \item We have an equality of schemes
 $\LAM_I(E_{\bull})=\OMM_I(E_{\bull})\cap X.$
\end{enumerate}
\end{proposition}
\begin{proof}  We
 first prove Part (1):
From the known description of $T(\OMM_I(E_{\bull}))_M$ as the space
of maps $\gamma:M\to V/M$ so that $\gamma(M\cap E_b)\subset
(E_b+M)/M$ for all $b$, we see that for $\gamma$ to also be in
$\home(M,M^{\perp}/M)$, the condition is
\begin{equation}\label{11}\gamma(M\cap E_b)\subset
((E_b+M)\cap M^{\perp}) /M \,\,\text{for any} \, b\in[2n].
\end{equation}
By Lemma ~\ref{Mper}, $E_b\cap M^{\perp}/E_b\cap M$ is of dimension
$|b\geq \tilde{I}\sqcup I|-|b\geq I|= |b\geq \tilde{I}|$ for any
$b\in [2n]$. Putting $b=i_a$, the condition given in Equation
\eqref{11} can be rewritten as
$$\gamma(E(M)_a)\subseteq
E(M^{\perp}/M)_{\lambda_a} \,\forall a\in[r].$$
The dimension of this vector space is clearly $|I\geq \tilde{I}|$.
This proves Part (1).
In particular,
\begin{equation}\label{13}
\dim \bigl(\home(M, M^{\perp}/M) \cap T(\LAM_I(E_{\bull}))_M
\bigr)\leq \dim \bigl( \home(M, M^{\perp}/M) \cap
T(\OMM_I(E_{\bull}))_M\bigr)=|I\geq \tilde{I}|= |I>\tilde{I}|,
\end{equation}
where the last equality follows since $I\cap \tilde{I}= \emptyset$.
We have (by the description of
$T(\OMM_I(E_{\bull}))_M$ given above)
\begin{equation}\label{12}
\begin{split}
&\phi(T(\LAM_I(E_{\bull}))_M)\subseteq \phi(T(\OMM_I(E_{\bull}))_M
\cap T(X)_M)\\
&\subseteq \{\gamma\in \Sym^2 M^*:\gamma(E_b\cap M, E_{b}^{\perp}\cap M)=0,\ \forall b\in[2n]\}\\
&=\{\gamma\in \Sym^2 M^*:\gamma(E_b\cap M, E_{2n-b}\cap M)=0,
\forall
b\in[2n]\}\\
&=\{\gamma\in \Sym^2 M^*: \gamma(E(M)_a,
E(M)_{t_a})=0,\forall  a\in[r]\}.\\
\end{split}
\end{equation}

(In the above, we have used $\dim (E_{2n-i_a}\cap
M)=t_a$. Furthermore, if $E_b\cap M=E_{b+1}\cap M$, then the
condition $\gamma(E_b\cap M, E_{2n-b}\cap M)=0$ implies the
condition $\gamma(E_{b+1}\cap M, E_{2n-b-1}\cap M)=0$.)

Moreover, the last space has dimension $\sym^2(I)$ by the following calculation.

As above, express the vector space under consideration 
in the more symmetric form
$$V_1=\{\gamma\in \Sym^2 M^*: \gamma(E_b\cap M, E_{2n-b}\cap M)=0,\forall
b\in[2n]\}.$$

Form the ``analogous'' space
$$V_2=\{\gamma\in \wedge^2 M^*: \gamma(E_b\cap M, E_{2n-b}\cap M)=0,\forall
b\in[2n]\}.$$

Clearly, $$V_1\oplus V_2=\{\gamma\in (M\tensor M)^*: \gamma((E_b\cap
M)\tensor (E_{2n-b}\cap M))=0,\ \forall b\in[2n]\},$$
 which, in turn,  can be
written as
$$\{\gamma\in \home(M,M^*): \gamma(E_b\cap M)\subset \ann(E_{2n-b}\cap
M),\forall b\in[2n]\},$$
where $\ann(E_{2n-b}\cap M)$ is the annihilator of $E_{2n-b}\cap M$ in $M^*$.
The last space
is of dimension $|I>\bar{I}|$. For
this note that $\ann(E_{2n-i_a}\cap M)$  is of dimension
$r-|I\leq 2n-i_a|=  |I > 2n-i_a|=|I > 2n+1-i_a|= |i_a> \bar{I}|$, since $I\cap
\bar{I}=\emptyset$.

Choose a basis of $M$ compatible with the filtration $\{E_b\cap M\}_{b\in [2n]}$.
Then, in this basis, the vector spaces $V_1$ and $V_2$ are subspaces of symmetric and
skew-symmetric matrices respectively. The difference $\dim V_1-\dim
V_2$ is the number of diagonal terms allowed in $V_1$, i.e.,
\[\dim V_1-\dim
V_2=|\{a\in [r]:E_{i_a}\cap M \not\subset E_{2n-i_a}\cap M\}|=|\{a\in [r]:i_a>2n-i_a\}|
=|I>n|.\]
Therefore, we conclude
$$\dim V_1 +\dim V_2\ =\ |I>\bar{I}|$$
$$\dim V_1 -\dim V_2\ =\ |I>n|.$$
Using Equation \eqref{4}, we obtain  $\dim V_1=\sym^2 I$ and $\dim
V_2=\wedge^2 I$. This proves the last assertion in Part (3).

Combining Equations ~\eqref{13} and ~\eqref{12}, and using Lemma
~\ref{gym}, we get
\begin{equation}\label{14}
\dim (T(\LAM_I(E_{\bull}))_M)\leq \sym^2(I)+ |I>\tilde{I}|.
\end{equation}
But, by Proposition ~\ref{morn2}, the above inequality 
 is an equality. Hence, all the inclusions and
inequalities in Equations ~(\ref{13})- (\ref{12}) are equalities.
This proves Parts (2)-(3).

By the definition, set theoretically
$\LAM_I(E_{\bull})=\OMM_I(E_{\bull})\cap X.$ Moreover, by
Parts  (2)-(3), the tangent space $T(\LAM_I(E_{\bull}))_M =
T(\OMM_I(E_{\bull}))_M \cap T(X)_M.$ This proves Part (4) and thus
completes the proof of the proposition.
\end{proof}

 Let $V$ be as in the beginning of this section  and let $r\leq n$. Fix $M\in
 \IG(r,V)$
 and  $I\in \FS(r,2n)$. Let $U=U_I(M)$ be the set of isotropic flags
$E_{\bull}$ on $V$ such that $M\in \LAM_I(E_{\bull})$. A flag
$E_{\bull}\in U$ induces a complete flag $E_{\bull}(M)$ on $M$, and
a complete isotropic flag $E_{\bull}(M^{\perp}/M)$ on $M^{\perp}/M$
(Lemma ~\ref{water}). The following lemma follows by choosing basis
elements appropriately.
\begin{lemma}\label{attract}
The map $U\to \Fl(M)\times \IFl(M^{\perp}/M)$ is  a surjective
fiber bundle with irreducible fibers, where $ \Fl(M)$ is the variety of all
the  complete flags on $M$ and $ \IFl(M^{\perp}/M)$ is the variety of all 
the isotropic flags on $ M^{\perp}/M$.
\end{lemma}

 Now, let $V$ be a $2r$-dimensional vector space
with a nondegenerate symplectic form containing $M$ of dimension $r$
as an isotropic
subspace. Let $I^1,\dots,I^s\in \FS(r,2r)$ be written as $I^j=\{i_1^j< \dots
< i_r^j\}$.
  Define, for any $j\in [s]$ and $a\in [r]$, $t^j_a=|\bar{I}^j\geq i^j_a|$.

\begin{lemma}\label{old3}
The following are
equivalent:
\begin{enumerate}
\item[($\alpha$)] $\prod_{j=1}^s[\bar{\LAM}_{{I^j}}]\neq 0\in H^{*}(\IG(r,2r)).$
\item[$(\beta)$] For some (and hence generic) complete flags  $F^1_{\bull}, \dots, F^s_{\bull}$ on
$M$, the vector space
\begin{equation}\label{vss23}
\{\gamma\in \sym^2 M^*\mid \gamma(F^j_a, F^j_{t^j_a})=0,\ a\in[r],
j\in [s]\}
\end{equation}
 is of the expected dimension
$$r(r+1)/2-\sum_{j=1}^s\cosym^2(I^j),$$
where $\cosym^2(I^j):=\frac{r(r+1)}{2}-\sym^2(I^j).$
\end{enumerate}
\end{lemma}
\begin{proof}
{\bf ($\alpha$)$\Rightarrow$ ($\beta$)}: Choose generic isotropic
flags $E^1_{\bull}, \dots, E^s_{\bull}$ on
$V$. Because of the assumption $(\alpha)$, $\cap_{j=1}^s
\LAM_{I^j}(E^j_{\bull})$ is nonempty and, by simultaneously translating each $E^j_{\bull}$
by a single element of $\Sp(2r)$, we can assume that $M$ belongs to this intersection.
Let
$F^1_{\bull}, \dots, F^s_{\bull}$ be the induced flags on $M$.

The vector space ~\eqref{vss23} is the tangent space to the scheme
theoretic intersection $\cap_{j=1}^s \LAM_{I^j}(E^j_{\bull})$ at $M$
(by Proposition ~\ref{morning} (3)). Again applying Proposition ~\ref{morning} (3),
the transversality of the
intersection implies that ($\beta$) holds.

Now assume ($\beta$). Choose (generic) complete flags $F^1_{\bull}, \dots, F^s_{\bull}$  on $M$ such that $(\beta)$ is satisfied. Now choose isotropic flags  $E^1_{\bull}, \dots, E^s_{\bull}$
on $V$ such that $M\in \cap_{j=1}^s \LAM_{I^j}(E^j_{\bull})$ and
the induced flags on $M$ are $F^1_{\bull}, \dots, F^s_{\bull}$ respectively.
This is possible by Lemma ~\ref{attract}.
Using ($\beta$) and Proposition ~\ref{morning} (3), we see that $\cap_{j=1}^s
\LAM_{I^j}(E^j_{\bull})$ is transverse at $M$. Therefore, using
standard facts from  intersection theory on a homogenous space, we see
that ($\alpha$) holds.
\end{proof}

\vskip2ex
Let $\{\bar{\epsilon}_i\}_{1\leq i\leq n}$ be the basis of $\frh^C$, where
$\bar{\epsilon}_i$ is the diagonal matrix
\[\text{diag}(0, \dots,0, 1,0, \dots, 0, -1,0,
\dots, 0),\] where $1$ is placed in the $i$-th slot and $-1$ in the $(2n+1-i)$-th slot.
Recall from Section 2 that we have identified $\frh^C$ with $\frh^B$ and denote
either of them by $\frh$. Let $\{{\epsilon}_i\}_{1\leq i\leq n}$ be the dual basis
(of $\frh^*$). Let $\rho^C$ (respectively, $\rho^B$) be half the sum of positive roots of
$\Sp(2n)$ (respectively, $\SO(2n+1)$). Then,
\[2\rho^B=(2n-1)\epsilon_1+(2n-3)\epsilon_2+\cdots + 3\epsilon_{n-1}+\epsilon_n,\]
and
\[2\rho^C=(2n)\epsilon_1+(2n-2)\epsilon_2+\cdots + 4\epsilon_{n-1}+2\epsilon_n.\]
Thus,
\[2(\rho^C-\rho^B)=\epsilon,\,\,\,\text{where}\, \epsilon:=\epsilon_1+\cdots +\epsilon_n.\]
Let $\{x_1^C, \dots, x_n^C\}$ be the basis
of $\frh^C$ as in Section 2 and similarly for  $x^B_r$. It is
easy to see that
\begin{align}\label{1}
 x_i^B&=x_i^C=\bar{\epsilon}_1+\cdots + \bar{\epsilon}_i, \,\,\,\text{if}\,\,\, 1\leq i\leq n-1
 \,\,\text{and}\notag\\
 x^B_n&=2x_n^C=\bar{\epsilon}_1+\cdots + \bar{\epsilon}_n.
 \end{align}

\begin{lemma}\label{lemma2}  For any $I\in \FS(r,2n)$,
  \[
w_I(\eps_{r+1} +\cdots +\eps_{n} ) = \eps_{j_1}+\cdots +\eps_{j_{n -r}},
  \]
for some $1\leq j_1<j_2<\cdots <j_{n -r}\leq n$.
  \end{lemma}

  \begin{proof}  This follows easily from the definition of $\FS(r,2n)$ by observing
  that the action of $W^C \subset S_{2n}$ on $\frh^*=\oplus_{i=1}^n\,\Bbb C \epsilon_i$
  is given by the standard action of $S_{2n}$ on $\oplus_{i=1}^{2n}\,\Bbb C \epsilon_i$
  and making the identification $\eps_j=-\eps_{2n+1-j}$, for any $1\leq j\leq 2n$.
  \end{proof}

\begin{lemma}  \label{lemma3} For any $w\in W^C$,
  \[
w\eps \in \pm\eps_1\pm \eps_2\pm\cdots\pm\eps_{n} .
  \]

Moreover,
  \begin{equation}\label{5}
\mu (w) = \# \text{ negative signs in the above}
= \frac{1}{2} \bigl(n -\langle w\eps ,\eps\rangle \bigr),
  \end{equation}
  where $\langle\,,\,\rangle$ is the normalized $W^C$-invariant form on $\frh^*$ given by
  $\langle\eps_i,\eps_j\rangle=\delta_{i,j}.$

 In particular, for any $I\in  \FS(r,2n)$,
\begin{equation}\label{4}\mu(w_I)=|I>n|.
\end{equation}
  \end{lemma}

  \begin{proof}  Since the symmetric group $S_{n}\subset W^C$ (acting as the permutation group
  of $\{\eps_1, \dots, \eps_n\}$) acts
trivially on $\eps$, the first part follows easily.  Applying $s_{n}$ to
$\pm\eps_1 \pm +\cdots\pm\eps_{n}$, the number of minus signs either
increases by 1 or decreases by 1, whereas the application of any $s_i$
to $\pm\eps_1 \pm +\cdots\pm\eps_{n}$ (for any $i<n$) keeps the number of minus
signs unchanged. Thus, the number of negative signs in $w\eps \leq \mu(w)$, for any $w\in W^C$.
Further, the last equality in Equation ~(\ref{5}) is trivially satisfied.
 Now, the longest element $w_0\in W^C$ makes all the signs in $w_0\eps$ negative and
$\mu (w_0)=n$.  From this  Equation ~(\ref{5}) follows for $w_0$ and hence for any $w\in W^C$. Otherwise,
taking a  reduced decomposition of $w$ and extending to a reduced decomposition of $w_0$,
we would get a contradiction.

 Equation ~(\ref{4}) follows from  Equation ~(\ref{5}) immediately.
  \end{proof}

  \begin{corollary} \label{mufunction} For any $I\in \FS(r, 2n)$,
  \[
(w_I^{-1}\eps) (x^B_r)= r-2\mu (w_I).
  \]
  \end{corollary}

  \begin{proof}
  \begin{align*}
(w_I^{-1}\eps) (x^B_r)&= \langle w_I^{-1}\eps ,\eps\rangle - \langle w_I^{-1}\eps
,\eps_{r+1}+\cdots +\eps_{n}\rangle\\
&= n -2\mu (w_I) -\langle\eps ,w_I(\eps_{r+1}+\cdots +\eps_{n} )\rangle,
\;\text{by Equation ~(\ref{5})}\\
&= n -2\mu (w_I) -(n -r), \;\text{by Lemma ~\ref{lemma2}}\\
&= r-2\mu (w_I).
  \end{align*}
  \end{proof}

Fix $I^1,\dots, I^s\in \FS(r, 2n)$ and define functions $\theta^C,
\theta^B:\FS(r,2n)\to \Bbb Z$ by
\[\theta^C(I)= \bigl(\chi^C_{w_I}-\sum_{j=1}^s\,\chi^C_{w_{I^j}}\bigr)(x^C_r),\]
and
\[\theta^B(I)= \bigl(\chi^B_{w_I}-\sum_{j=1}^s\,\chi^B_{w_{I^j}}\bigr)(x^B_r),\]
where $\chi^C_{w_I}:=(\rho^C+w_I^{-1}\rho^C)$ and  $\chi^B_{w_I}$ is defined
similarly.
\begin{lemma} \label{sosp} For $r<n$ and any $I\in \FS(r,2n)$,
\begin{equation*} \theta^C(I)=\theta^B(I)+\bar{\mu}(I)-\sum_{j=1}^s\bar{\mu}(I^j),
\end{equation*}
where $\bar{\mu}(I):=r-\mu(w_I)$.
Similarly, for $r=n$,
\begin{equation*} 2\theta^C(I)=\theta^B(I)+\bar{\mu}(I)-\sum_{j=1}^s\bar{\mu}(I^j).
\end{equation*}
\end{lemma}
\begin{proof} Assume first that $r<n$. In this case
\begin{align}
 \theta^C(I)-\theta^B(I)&=\frac{1}{2}\bigl((\eps+w_I^{-1}\eps)-\sum_{j=1}^s\,
(\eps+w_{I^j}^{-1}\eps)\bigr)(x_r^C)\notag\\
&=r-\mu(w_I)-\sum_{j=1}^s(r-\mu(w_{I^j})),\,\, \text{by\,Corollary}\,
 ~\ref{mufunction}\notag\\
 &=\bar{\mu}(I)-\sum_{j=1}^s\bar{\mu}(I^j)\notag.
 \end{align}
 Now, we consider the case $r=n$. In this case,
 \begin{align}
2\theta^C(I)-\theta^B(I)&=\bigl((\eps+w_I^{-1}\eps)-\sum_{j=1}^s\,
(\eps+w_{I^j}^{-1}\eps)\bigr)(x_n^C)\notag\\
 &=\bar{\mu}(I)-\sum_{j=1}^s\bar{\mu}(I^j)\notag.
 \end{align}
 This proves the lemma.
\end{proof}
\begin{lemma}\label{old2}
Let $I^1,\dots,I^s\in \FS(r,2r)$ be such that
$$\prod_{j=1}^s[\bar{\LAM}_{{I^j}}]\neq 0\in
H^{*}(\IG(r,2r)).$$
Then,
$$\sum_{j=1}^s \bar{\mu}(I^j) \geq r -(\dim \IG(r,2r)- \sum_{j=1}^s  \codim(\LAM_{I^j})).$$
Observe that, by Equation ~(\ref{4}),
 \begin{equation}\label{-3}\bar{\mu}({I^j})=|r\geq I^j|.
\end{equation}
\end{lemma}
\begin{proof}
Let $I\in \FS(r,2r)$ be such that $[\bar{\LAM}_I]$ appears with
nonzero coefficient in the product
$\prod_{j=1}^s[\bar{\LAM}_{{I^j}}]$. The maximal parabolic subgroup
$P_r^C$ is minuscule and hence by [BK, Lemma 19], $\theta^C(I)=0.$
Thus, by Lemma ~\ref{sosp},
\begin{equation}\label{-2}\theta^B(I)= -\bar{\mu}(I)+\sum_{j=1}^s\,\bar{\mu}(I^j).\end{equation}
But, clearly,
\[\mu(w_I)\leq \dim \LAM_I=\dim \IG(r,2r)-\sum_{j=1}^s\codim(\LAM_{I^j}).\]
Hence,
\begin{equation}\label{7}
\bar{\mu}(I)=r-\mu(w_I)\geq r-\dim \IG(r,2r)+\sum_{j=1}^s
\codim(\LAM_{I^j}).
\end{equation}
Also, by Theorem ~\ref{grain} and [BK, Proposition 17(a)],
\begin{equation}\label{8}
\theta^B(I)\geq 0.\end{equation}
Combining Equations ~(\ref{-2})- (\ref{8}), we get the lemma.

\vskip1ex

{\bf An alternative proof of the lemma:}
The following is a quick (though less conceptual) way to prove the
lemma. Let $I^o=\{r,r+2,\dots,2r\}$. Then, by Proposition ~\ref{morn2},
$\codim\LAM_{I^o}=1$ and $[\bar{\LAM}_{I^o}]$ is
ample on $\IG(r,2r)$. So, by
cupping with a sufficient number of  $[\bar{\LAM}_{I^o}]$, we are reduced to the case
$\dim \IG(r,2r)- \sum_j \codim(\LAM_{I^j})=0$.

Suppose now by way of contradiction that  $\sum_{j=1}^s
\bar{\mu}(I^j) <r$. Choose generic isotropic flags $E^1_{\bull},
\dots, E^s_{\bull}$ on a $2r$-dimensional vector space $V$ with a
nondegenerate symplectic form, and let $M\in \cap_{j=1}^s
\LAM_{I^j}(E^j_{\bull})$. Let $F^j_{\bull}$ be the induced flags on
$M$. Let $t^j_a:=|\bar{I}^j\geq i^j_a|$. Since the intersection
 $\cap_{j=1}^s \LAM_{I^j}(E^j_{\bull})$ is transverse at $M$, by Proposition
 ~\ref{morning} (3),
 there are no nonzero symmetric bilinear forms $\phi$ on $M$, so that
$$\phi(F^j_a,F^j_{t^j_a})=0,\ j\in [s],\ a\in [r].$$
Let $M':=\sum_{j=1}^s F^j_{\bar{\mu}(I^j)}$, which is a proper subspace of $M$
since, by assumption, $\sum_{j=1}^s\bar{\mu} (I^j)<r.$ Thus, $\Sym^2
(M/M')^*\hookrightarrow \Sym^2 M^*$. It is easy to see that any nonzero
element $\phi$ in $\Sym^2(M/M')^*$ satisfies the forbidden
possibility above (since at least one of $a$ or $t_a^j \leq \bar{\mu}(I^j)$
for any $a\in [r]$ and $j\in [s]$).
\end{proof}

\begin{lemma}\label{oldie}
Let $I^1,\dots,I^s\in \FS(r,2n)$ be such that
$$\prod_{j=1}^s[\bar{\LAM}_{{I^j}}]\neq 0\in
H^{*}(\IG(r,2n)).$$  Then, the  following inequalities hold:
\begin{equation}\label{inegalite1}
r(r+1)/2 -\sum_{j=1}^s\cosym^2(I^j)\geq 0.
\end{equation}
\begin{equation}\label{inegalite2}
\sum_{j=1}^s \bar{\mu}(I^j) \geq r -\bigl(\dim \IG(r,2r)- \sum_{j=1}^s
\cosym^2(I^j)\bigr),
\end{equation}
and
\begin{equation}\label{inegalite3}
r(r-1)/2 - \sum_{j=1}^s
\co\wedge^2(I^j)\geq 0,
\end{equation}
where recall that $\sym^2(I)$ and $\wedge^2(I)$ are defined above Proposition
~\ref{morn2}, $\cosym^2(I)$ is defined in Lemma ~\ref{old3}, and
$$\co \wedge^2(I):=r(r-1)/2-\wedge^2(I).$$
\end{lemma}
\begin{proof} Choose generic isotropic flags $E^1_{\bull}, \dots, E^s_{\bull} $ on
$V$, and let $M\in \cap_{j=1}^s \LAM_{I^j}(E^j_{\bull}).$ Let
$F^j_{\bull}$ be the induced flags on $M$. Since the intersection
 $\cap_{j=1}^s \LAM_{I^j}(E^j_{\bull})$ is transverse at $M$, the images of the tangent spaces of $ \LAM_{I^j}(E^j_{\bull})$ via the map $\phi: 
T(\IG(r,2n))_M\to
\sym^2 M^*$ (defined in Lemma ~\ref{gym})  are transverse in $\sym^2 M^*$,
 as can be easily seen (see [PS] for a similar argument).  This gives us
Equation ~(\ref{inegalite1}) in view of  Proposition ~\ref{morning} (3).

 Let  $\beta_j$ be the order preserving  bijection of $I^j\sqcup \bar{I}^j$ with $[2r]$
and let $I_o^j$ be the subset $\beta(I^j)$ of $[2r]$ (of cardinality  $r$).
  Hence, we see using Lemma ~\ref{old3} and Proposition ~\ref{morning} that
$$\prod_{j=1}^s[\bar{\LAM}_{{{I_o^j}}}]\neq 0\in
H^{*}(\IG(r,2r)).$$
Observe next that $|n\geq I^j|=|r\geq I^j_o|$ and hence
 $|I^j>n|=|I^j_o>r|$. Thus, by Equation ~(\ref{-3}),
\[\bar{\mu}(I^j)=\bar{\mu}(I^j_o)\,\,\text{and}\,\, \sym^2(I^j)=\sym^2(I^j_o).\]

Now, we can use  Lemma ~\ref{old2} and Proposition
~\ref{morn2} to
 obtain Equation ~(\ref{inegalite2}).  Equation
 ~(\ref{inegalite3}) follows from Equation ~(\ref{inegalite2}) and the
 equality
 $$\cosym^2(I^j)-\co\wedge^2(I^j)=r-\bigl(\sym^2(I^j)-\wedge^2(I^j)\bigr)=
\bar{\mu}(I^j).$$
 \end{proof}

\subsection{Isotropic flags and transversality}\label{three3}
For any $r\leq m \leq n$ define
$$\ml(m,r,2n)=\{X\in \Gr(m,V):\dim(X\cap X^{\perp})=r\},$$
where $V$ is as in the beginning of Section 3.  
\begin{lemma}\label{compr}
\begin{enumerate}
\item $\Sp(2n)$ acts transitively on $\ml(m,r,2n)$.
\item $\ml(m,r,2n)$ is a fiber bundle over $\IG(r,V)$ via the map $\pi$
defined by $X\mapsto M=X\cap X^{\perp}$.
\item  For a fixed $m$, the union of the subvarieties $\ml(m,r,2n)$ over all choices
of $0\leq r\leq m$ is  $\Gr(m,V)$.
\item The action of $\Sp(2n)$ on
$\Gr(m,V)$ has only finitely many orbits.
\end{enumerate}
\end{lemma}
\begin{proof}
Let $X,X'\in \ml(m,r,2n)$, $M=X\cap X^{\perp}$ and $M'=X'\cap
X'^{\perp}$. We need to produce  $g\in \Sp(2n)$ such that
$gX=X'$. Since $\Sp(2n)$ acts transitively on $\IG(r,2n)$,
we may assume that $M=M'$. Let
$$G_M:= \{g\in\Sp(2n): gM=M\}.$$
We easily see that $G_M$ surjects onto the symplectic group of
$M^{\perp}/M$ by choosing subspaces $N$ and $Y$ of $V$ such that
$V=M\oplus N\oplus Y, \dim N=\dim M$ and $\langle
M\oplus N,  Y \rangle=0. $ Now, $\pi^{-1}(\{M\})$ can be
identified with the
set of subspaces $\bar{X}\subset M^{\perp}/M$ of dimension $m-r$ which are nondegenerate with
 respect to the induced (nondegenerate) symplectic
form on $M^{\perp}/M$. This reduces us to the case of $r=0$ (for the
proof of the first part of the lemma). So, we need to show that $\Sp(2n)$ acts transitively
on the set of subspaces $N$ of a given dimension $m$ such that $N\cap N^{\perp}=(0)$. Let $X$ be one such
subspace and $Y= X^{\perp}$. Clearly, $V= X\oplus Y$ and the induced
symplectic forms are nondegenerate on $X$ and $Y$. Choose 
``symplectic bases'' $\{e_i, f_i\}$ of $X$ and $\{u_j, v_j\}$ of $Y$.
Map them appropriately to the standard basis vectors in $V$. This transformation is in $\Sp(2n)$.
This proves Part (1). Part (2) follows from the above description of
 $\pi^{-1}(\{M\})$. Part (3) of course is clear and Part (4) follows  from Parts (1)
 and (3).
\end{proof}
For an isotropic flag $E_{\bull}$ on $V$, and subsets $A\in S(m,2n)$ and $I\in \FS(r,2n)$,
 define the  subvariety:
$$\LAM_{I,A}(E_{\bull})=\{X\in\Gr(m,V):\ X\in \OMM_A(E_{\bull}), X\cap
X^{\perp}\in \LAM_{I}(E_{\bull})\}\subset \ml(m,r,2n).$$

\begin{lemma}\label{third}
Set $L:=A\setminus (I\sqcup \bar{I})$ and $T:=[2n]\setminus (I\sqcup
L\sqcup \bar{I})$. Assume that $\LAM_{I,A}(E_{\bull})$ is nonempty.
Then, it is irreducible and
$$\dim\LAM_{I,A}(E_{\bull})=|I> L| + |I> T|+ \Sym^2(I) + |L> T|.$$
Moreover, $I\subset A$ and $\bar{I}\cap A = \emptyset$.
\end{lemma}
\begin{proof}
This follows from considering the fibration
$\pi_{I,A}:\LAM_{I,A}(E_{\bull})\to\LAM_I(E_{\bull})$, where
$\pi_{I,A}$ is the restriction of $\pi$ to $\LAM_{I,A}(E_{\bull})$.
(This is a surjective fibration because of the $B^C$-equivariance.)
By Proposition \ref{morn2}, the dimension of $\LAM_I(E_{\bull})$ is
$$ |I> L\sqcup
T|+\Sym^2(I) =|I> L|+ |I> T|+ \Sym^2(I) .$$ So, we  just need to
prove that the fiber dimension of $\pi_{I,A}$ is $|L>T|$.

Let $M\in \LAM_I(E_{\bull})$, and $\bar{X}\in \Gr(m-r,M^{\perp}/M)$
be such that the induced form on $\bar{X}$ is nondegenerate. Then,
$X:=q^{-1}(\bar{X})\in \LAM_{I,A}(E_{\bull})$ if and only if  $X\in
\OMM_A(E_{\bull})$, where $q: M^{\perp} \to  M^{\perp}/M$ is the
quotient map. By definition, $X\in \OMM_A(E_{\bull})$ if and only
if $\dim(E_{a}\cap X)=|a\geq A|$ for all $a$. Now, by Lemma
~\ref{Mper}, the image of $E_a\cap M^{\perp}$ in $M^{\perp}/M$ is of
dimension $|a\geq [2n]\setminus\bar{I}|-|a\geq I|=|a\geq \tilde{I}|$
(see the footnote $^{(1)}$). Moreover, for any subspace $X$ of
$M^{\perp}$ containing $M$,
$$\dim(E_{a}\cap X)=\dim (E_a\cap M) + \dim ((E_a +M/M) \cap
\bar{X}).$$ Thus, for any $\bar{X}\in \Gr(m-r,M^{\perp}/M),
X:=q^{-1}(\bar{X})$ belongs to $ \LAM_{I,A}(E_{\bull})$ if and only
if
\begin{equation}\label{15}
\dim (E(M^{\perp}/M)_{\varphi(a)}\cap \bar{X})= |a\geq A|-|a\geq
I|,\,\, \forall a\in [2n],
\end{equation} where $\varphi(a):=|a\geq \tilde{I}|$.

We next claim that $I\subset A$. Write $A=\{1\leq a_1< \cdots < a_m\leq 2n\}$. If possible,
let $a\in I$ be such that $a_i< a <a_{i+1},$ for some $0\leq i\leq m$
(where we set $a_0=0$ and $a_{m+1}=2n+1$). Then,
\[|a\geq A| - |a\geq I|< |(a-1)\geq A| - |(a-1)\geq I|.\]
By Equation ~(\ref{15}), this gives
\[\dim (E(M^{\perp}/M)_{\varphi(a)}\cap \bar{X}) < 
\dim (E(M^{\perp}/M)_{\varphi(a-1)}\cap \bar{X}),
\]
which is a contradiction. Hence the claim $I\subset A$ is established.

We further claim that $\bar{I}\cap A=\emptyset$. If possible, take $a\in I$ such that
$a^*:=2n+1-a\in A$. Then,
\[|a^*\geq A| - |a^*\geq I|> |(a^*-1)\geq A| - |(a^*-1)\geq I|.\]
Thus, by Equation ~(\ref{15}),
\[\dim (E(M^{\perp}/M)_{\varphi(a^*)}\cap \bar{X}) >
 \dim (E(M^{\perp}/M)_{\varphi(a^*-1)}\cap \bar{X}).
\]
However, $\varphi(a^*)=\varphi(a^*-1)$, which is a contradiction.
This proves that $\bar{I}\cap A=\emptyset$. Thus, for any
$\bar{X}\in \Gr(m-r,M^{\perp}/M), X=q^{-1}(\bar{X})$ belongs to $
\LAM_{I,A}(E_{\bull})$ if and only if  for any $1 \leq p\leq 2n-2r$,
\begin{equation}\label{16}
\dim (E(M^{\perp}/M)_p\cap \bar{X})= |j_p\geq L|,
\end{equation}
where $\tilde{I}=\{1\leq j_1< \cdots <j_{2n-2r}\leq 2n\}$. This
shows that the fiber of $\pi_{I,A}$ is an open subset of a Schubert
cell in $\Gr(m-r,2n-2r)$ of dimension $|L> T|$ since
$T:=\tilde{I}\setminus L$. In particular, $ \LAM_{I,A}(E_{\bull})$
is irreducible. This proves the lemma.
\end{proof}

\begin{lemma}\label{fourth}
\begin{enumerate}
\item The dimension of $\ml(m,r,2n)$ is $\dim(\Gr(m,2n))-r(r-1)/2$.
\item  If nonempty, $\LAM_{I,A}(E_{\bull})$ is  of
dimension no greater than $\dim(\OMM_A(E_{\bull}))-\wedge^2(I).$
\end{enumerate}
\end{lemma}

\begin{proof} By Lemma \ref{compr},
$\ml(m,r,2n)$ is a fiber bundle over $\IG(r,V)$. Moreover, as in the proof of Lemma
~\ref{compr}, the fiber is an
open subset of $\Gr(m-r,2n-2r)$. Therefore,  $\dim \ml(m,r,2n)$ equals
\begin{align*}
\dim \IG(r,V) + (m-r)(2n-2r-(m-r))&= r(4n-3r+1)/2
+(m-r)(2n-m-r)\\
&=\dim(\Gr(m,2n))-r(r-1)/2.
\end{align*}
This proves the part (1) of the lemma.

We next prove part (2). By Lemma \ref{third},
\begin{align}\label{10}\dim\LAM_{I,A}(E_{\bull})&= |I> L| + |I> T|+ \Sym^2(I) + |L> T|\notag\\
&=|I> T| +|L> \bar{I}|+|L> T| +|I> \bar{I}|-\wedge^2(I) +
|I> L|-|L> \bar{I}|\notag\\
&=\dim \OMM_A(E_{\bull})-\wedge^2(I) +|I> L|- |L> \bar{I}|.
\end{align}
The lemma  will therefore follow from the inequality
\begin{equation}
\label{9}
|I>L|\leq |L>\bar{I}|,
\end{equation} which is proved below.

We first show
\begin{equation}\label{19}|a\geq L|\leq |L\geq a^*| \,\,\text{for any}\, a\in [2n],\,\,
\text{where}\,
a^*:=2n+1-a.
\end{equation}
Take $X\in \LAM_{I,A}(E_{\bull})$ and let $\pi(X)=M$.  Let $Y:=X\cap
E_a\supset M\cap E_a$. We look at the image of $X\to Y^*$  (induced
by the symplectic form), where $Y^*$ is the dual of $Y$. The image
 has dimension
$$\dim Y- \dim (Y\cap M)=|a\geq A|-|a\geq I|=|a\geq L|.$$ But, clearly,  $(X\cap E_{a^{*}-1})+M$ goes to
zero under the map $X\to Y^*$ and hence the image of $X\to Y^*$ is
of dimension no greater than $|L\geq a^*|$.
This proves Equation ~(\ref{19}).
Hence,
$$|I\geq L|\leq |L\geq \bar{I}|.$$
But since $I\cap L=\emptyset$ and also  $\bar{I}\cap L=\emptyset$, we get
Equation ~(\ref{9}) and hence Part (2) of the lemma  is proved.
\end{proof}

\subsection{Proof of Theorem ~\ref{wilson1}}
Choose isotropic flags $\{E_{\bull}^j\}_{1\leq j\leq s}$ such that
the intersection $\cap_{j=1}^s{{\LAM}}_{I^j}(E^j_{\bull})$ is
transverse and dense in $\cap_{j=1}^s \bar{\LAM}_{I^j}(E^j_{\bull})$
for all $I^j \in \FS(r,2n)$ and all $1\leq r\leq m$ 
(cf. [BK, Proposition 3]). For any
irreducible component $C$ of $\cap_{j=1}^s \OMM_{A^j}(E^j_{\bull})$,
there exists a dense open subset $U$ of $C$ and an $r$ together with
subsets $\{I^j\}_{1\leq j \leq s}\subset \FS(r,2n)$  such that
$U\subset  \ml(m,r,2n)$ and $\pi(U)\subset \cap_{j=1}^s\LAM_{I^j}
(E^j_{\bull})$.

It suffices to show that the dimension of any irreducible component
of $\cap_{j=1}^s \LAM_{I^j,A^j}(E^j_{\bull})$ is no greater than
$\dim\Gr(m,2n)-\sum_{j=1}^s\codim(\OMM_{A^j})$.

Take (generic) isotropic flags  $\{E_{\bull}^j\}_{1\leq j\leq s}$
such that the intersection $\cap_{j=1}^s\LAM_{I^j,
A^j}(E^j_{\bull})$ in $\ml(m,r,2n)$ is proper for all the choices of
$I^j \in \FS(r,2n)$ and
 $A^j\in S(m,2n)$. This is possible since $\Sp(2n)$ acts transitively on
$\ml(m,r,2n)$.

 Therefore, using Lemma ~\ref{fourth}, any irreducible
component of  $\cap_{j=1}^s\LAM_{I^j, A^j}(E^j_{\bull})$ has dimension no greater than
\begin{align*}
\dim\Gr(m,2n)-&(r(r-1)/2)-\sum_{j=1}^s\Bigl(\dim\Gr(m,2n)-r(r-1)/2-(\dim
\OMM_{A^j}-\wedge^2(I^j))\Bigr)\notag\\
&=\dim\Gr(m,2n)-\bigl(\sum_{j=1}^s\codim\OMM_{A^j}\bigr)
-\Bigl(r(r-1)/2- \sum_{j=1}^s\co\wedge^2(I^j)\Bigr).
\end{align*}
Using Equation ~(\ref{inegalite3}), this
last quantity is no greater than
$$\dim\Gr(m,2n)-\sum_{j=1}^s\codim\OMM_{A^j},$$ as desired.
 This finishes the proof of the theorem. \qed

\subsection{Analogue of Theorem  ~\ref{wilson1} for $\SO(2n+1)$}
We now prove the analogue of Theorem  ~\ref{wilson1} for $\SO(2n+1)$.
 The proof is very similar and hence we will only give a brief sketch. First we need the
 following preparation. As in Subsection 2.3, $V':=\Bbb C^{2n+1}$ is
 equipped with the  quadratic form $Q$.

For any $r\leq m\leq n$ define
$$\mo(m,r,2n+1)=\{X\in \Gr(m,V'):\dim(X\cap X^{\perp})=r\}.$$
We have the following analogue of Lemma ~\ref{compr} for $\SO(2n+1)$. Its proof is similar and hence is omitted.
\begin{lemma}
\begin{enumerate}
\item $\SO(2n+1)$ acts transitively on $\mo(m,r,2n+1)$.
\item $\mo(m,r,2n+1)$ is a fiber bundle over $\OG(r,2n+1)$ via the map $\pi'$
defined by
$X\mapsto M=X\cap X^{\perp}$.

\item For a fixed $m$, the union of the subvarieties $\mo(m,r,2n+1)$ over all choices of
$0\leq r\leq m$ is clearly $\Gr(m,V')$.
In particular, the action of $\SO(2n+1)$ on $\Gr(m,V')$ has only finitely
many orbits.
\end{enumerate}
\end{lemma}

 For an isotropic flag $E'_{\bull}$ on $V'$, and subsets
$A\in  S(m, 2n+1)$ and $J \in \FS'(r, 2n+1)$ (cf. Subsection 2.3), define the
 subvariety of $\mo(m,r,2n+1)$:
$$\Psi_{J,A}(E'_{\bull})=\{X\in\Gr(m,V'):\ X\in \OMM_A(E'_{\bull}), X\cap
X^{\perp}\in \Psi_{J}(E'_{\bull})\}.$$

\begin{lemma}\label{fourthy}
\begin{enumerate}
\item The dimension of $\mo(m,r,2n+1)$ is $\dim(\Gr(m,2n+1))-r(r+1)/2$.
\item If nonempty, $\Psi_{J,A}(E'_{\bull})$ is irreducible of
dimension no greater than $\dim(\OMM_A(E'_{\bull}))-\sym^2(J)$.
\end{enumerate}
\end{lemma}

\begin{proof} The fiber of the bundle
$\pi:\mo(m,r,2n+1) \to \OG(r,V')$ is an
open subset of the set of choices of  $m$-dimensional spaces
containing a fixed $r$-dimensional subspace and contained in the
perpendicular space of this $r$-dimensional space. Therefore, by a
simple calculation, $\dim \mo(m,r,2n+1)$ equals
\begin{align*}\dim(\OG(r,V')) + (m-r)&(2n+1-2r-(m-r))\\
&= \bigl(r(r-1)/2
+r(2n+1-2r)\bigr)
+(m-r)(2n+1-m-r)\\
 &=\dim(\Gr(m,V'))-r(r+1)/2.
\end{align*}
This proves Part (1).

We next observe that as in the symplectic case (Lemma ~\ref{third}),
if nonempty,
$$\dim\Psi_{J,A}(E'_{\bull})=|J> L'| + |J> T'|+ \wedge^2(J) +
 |L'> T'|,$$
where
 $L':=A\setminus (J\sqcup \bar{J}')$,
$T':=[2n+1]\setminus
 (J\sqcup L'\sqcup \bar{J}')$ and
$\bar{J}'$ is defined in Subsection 2.3.

Similar to Equation ~(\ref{10}), we have
\[\dim\Psi_{J,A}(E'_{\bull})=\dim(\OMM_A(E'_{\bull}))-\sym^2(J) +|J > L'|-|L'>\bar{J}'|.
\]
The second part of the lemma will therefore follow from the following inequality whose proof is
similar to that of Equation ~(\ref{9}).
\begin{equation}|J> L'|\leq |L'>\bar{J}'|.
\end{equation}
\end{proof}
\begin{theorem}\label{newwilson2}
 Let $A^1,\dots,A^s$ be
subsets of $[2n+1]$ each of cardinality $m$. Let
${E'}^1_{\bull}, \dots, {E'}^s_{\bull}$ be isotropic flags on $V'=\Bbb
C^{2n+1}$ in general position. Then, the intersection $\cap_{j=1}^s
\OMM_{A^j}({E'}^j_{\bull})$  of subvarieties of $\Gr(m,V')$ is
 proper.
\end{theorem}

\begin{proof}
 For any irreducible component $C'$ of
$\cap_{j=1}^s \OMM_{A^j}({E'}^j_{\bull})$, there exists a dense open
subset $U'\subset C'$ and an $r$ together with subsets
$\{J^j\}_{1\leq j\leq s}\subset \FS'(r,2n+1)$ such that  $U'\subset
\mo(m,r,2n+1)$ and $\pi'(U')\subset
\cap_{j=1}^s\Psi_{J^j}({E'}^j_{\bull}).$

It suffices to show that the dimension of any irreducible component
of $\cap_{j=1}^s \Psi_{J^j,A^j}({E'}^j_{\bull})$ is no greater than
$\dim(\Gr(m,2n+1))-\sum_{j=1}^s\codim(\OMM_{A^j})$.
 The dimension of $\cap_{j=1}^s
\Psi_{J^j,A^j}({E'}^j_{\bull})$ is no greater than
\begin{align*}&\dim(\Gr(m,2n+1))-r(r+1)/2\\
-&\sum_{j=1}^s\Bigl(\dim(\Gr(m,2n+1))-r(r+1)/2-(\dim
(\OMM_{A^j})-\sym^2(J^j))\Bigr)\\
&=\dim(\Gr(m,2n+1))-r(r+1)/2 - \sum_{j=1}^s\bigl(\codim(\OMM_{A^j})-
 \cosym^2(J^j)\bigr).
\end{align*}
Using Theorem ~\ref{grain} and  Equation ~(\ref{inegalite1}), 
this last quantity is no greater than
$\dim(\Gr(m,2n+1))-\sum_{j=1}^s\codim\OMM_{A^j}$ as desired.
\end{proof}

As an immediate consequence of the above theorem (just as in the case of $\Sp(2n)$), we get the following:

\begin{corollary}\label{frasier'} Let $1\leq m\leq n$ and let
 $J^1,\dots, J^s\in \FS'(m,2n+1)$  be such that
$$\prod_{j=1}^s[\bar{\Psi}_{{J^j}}]\neq 0\in H^{*}(\OG(m,2n+1)).$$
 Then,  $\prod_{j=1}^s[\OMC_{{J^j}}]\neq 0\in H^{*}(\Gr(m,2n+1)).$
\end{corollary}

\section{Comparison of eigencone for $\Sp(2n)$ with that of $\SL(2n)$}

Let $G$ be a connected (complex) semisimple group. Choose  a maximal compact subgroup
  $K$ of $G$ with Lie algebra $\frk$. Then,
there is a natural homeomorphism
$C:\frk/K\to \frh_{+}$, where $K$ acts on $\frk$ by the adjoint representation
and $ \frh_{+}$ is the positive Weyl chamber in $\frh$ as in Section 2. The inverse map
$C^{-1}$ takes any $h\in \frh_+$ to the $K$-conjugacy class of $ih$.

For a positive integer $s$,  the {\it eigencone} is defined as the cone: $\Gamma(s,K):=$
$$\{(h_1,\dots,h_s)\in\frh_{+}^s\mid \exists (k_1,\dots,k_s)\in \mathfrak k^s \text{:} \sum_{j=1}^s k_j=0\,\,\text{and }\, C(k_j)=h_j \forall j=1,\dots,s\}.$$

Given a standard maximal parabolic subgroup $P$, let $\omega_P$ denote the corresponding fundamental weight, i.e., $\omega_P(\alpha_i^\vee)=1$, if $\alpha_i \in \Delta\setminus \Delta(P)$ and $0$ otherwise, where $\Delta(P)$
is the set of simple roots for the Levi subgroup $L$ of $P$ containing $H$.
 Then,  $\omega_P$ is invariant under the Weyl group $W_P$ of $P$.

We recall the following theorem from [BeSj].
\begin{theorem}\label{eigen}  Let $(h_1,\dots,h_s)\in\frh_{+}^s$.  Then, the following
are equivalent:

(a) $(h_1,\dots,h_s)\in\Gamma(s,K)$.

(b)
 For every standard maximal parabolic subgroup $P$ in $G$ and every choice of
  $s$-tuples  $(w_1, \dots, w_s)\in (W^P)^s$ such that
$$[\bar{\Lambda}^P_{w_1}]\cdot \, \dots \,\cdot [\bar{\Lambda}^P_{w_s}] =d
 [\bar{\Lambda}^P_{e}] \in H^*(G/P), \,\text{for some nonzero}\, d, $$
the following inequality holds:
$$\omega_P(\sum_{j=1}^s\,w_j^{-1}h_j)\leq 0.$$

In fact, assume that (a) is satisfied, i.e.,
   $(h_1,\dots,h_s)\in\Gamma(s,K)$. Then,
 for every standard maximal parabolic subgroup $P$ in $G$ and every choice of
  $s$-tuples  $(w_1, \dots, w_s)\in (W^P)^s$ such that
$$[\bar{\Lambda}^P_{w_1}]\cdot\, \dots \,\cdot [\bar{\Lambda}^P_{w_s}]
\neq 0 \in H^*(G/P),$$
the following inequality holds:
$$\omega_P(\sum_{j=1}^s\,w_j^{-1}h_j)\leq 0.$$
\end{theorem}

Recall that $\frh^C_+$ (respectively, $\frh^B_+$) is the dominant chamber in
the Cartan subalgebra of $\Sp(2n)$ (respectively, $\SO(2n+1)$) as in Section 2.

The following theorem is our main result on the comparison of the eigencone for $\Sp(2n)$
 with that of  $\SL(2n)$ (and also for  $\SO(2n+1)$
 with that of  $\SL(2n+1)$).
\begin{theorem}\label{woodrow1}
\begin{enumerate}
\item[(a)]  For $h_1,\dots,h_s\in \frh_{+}^{C}$, $$(h_1,\dots,h_s)\in
\EV(s,\Sp(2n))\Leftrightarrow (h_1,\dots,h_s)\in
\EV(s,\SU(2n)).$$
\item[(b)] For $h_1,\dots,h_s\in \frh_{+}^{B}$,
$$(h_1,\dots,h_s)\in
\EV(s,\SO(2n+1))
\Leftrightarrow(h_1,\dots,h_s)\in
\EV(s,\SU(2n+1)).$$
\end{enumerate}

(Observe that by Section 2, $\frh_{+}^{C} \subset\frh_{+}^{A_{2n-1}}$
and $\frh_{+}^{B} \subset\frh_{+}^{A_{2n}}$.)
\end{theorem}
\begin{proof} Clearly, $\EV(s,\Sp(2n))\subseteq \EV(s,\SU(2n))$. Conversely, we
need to  show that if ${\bf h}=(h_1,\dots,h_s)\in (\frh_{+}^{C})^s$
is such that ${\bf h}\in  \EV(s,\SU(2n))$, then  ${\bf h}\in  \EV(s,\Sp(2n))$. Take any
$1\leq m\leq n$ and any $I^1, \dots, I^s \in \FS(m,2n)$ such that
$$[\bar{\LAM}_{I^1}]\cdot\, \dots \,\cdot [\bar{\LAM}_{I^s}]=
d[\bar{\LAM}_e] \in
 H^*(\IG(m,2n))\,  \,\text{for some nonzero}\, d .$$
By Corollary ~\ref{frasier},
$$[\OMC_{I^1}]\cdot\, \dots \,\cdot [\OMC_{I^s}]\neq 0\in
H^{*}(\Gr(m,2n)).$$
In particular, by Theorem ~\ref{eigen} applied to $\SU(2n)$,
\[\omega_m(\sum_{j=1}^sv_{I^j}^{-1}h_j)\leq 0,\]
where $\omega_m$ is the $m$-th fundamental weight of $\SL(2n)$ and
$v_{I^j}\in S_{2n}$ is the element associated to $I^j$ as in Subsection 2.1.
It is easy to see that the  $m$-th fundamental weight $\omega_m^C$ of $\Sp(2n)$
is the restriction of $\omega_m$ to $\frh^C$. Moreover, even though the elements
$v_{I^j}\in S_{2n}$ and $w_{I^j}\in W^C$ are, in general, different, we still have
$$\omega_m(v_{I^j}^{-1}h_j)=\omega_m^C(w_{I^j}^{-1}h_j).$$

Applying Theorem \ref{eigen} for $\Sp(2n)$, we get the (a)-part of
the theorem.

The proof for $\SO(2n+1)$ is similar. (Apply Corollary ~\ref{frasier'} instead of Corollary ~\ref{frasier}.)
\end{proof}

\section{A basic transversality result}
Let $M$ be a $r$-dimensional space and $\mf=(F^1_{\bull},\dots,
F^s_{\bull})$
  an $s$-tuple of complete flags on $M$. As earlier, let $V$ be a $2n$-dimensional vector space
equipped with a nondegenerate symplectic form, and
$\mg=(G^1_{\bull},\dots, G^s_{\bull})$ an $s$-tuple of (complete)
isotropic flags on $V$.
 Let $\mu=(\mu^1,\dots,\mu^s)$, where $\mu^j$ is an ordered sequence
 with $r$ elements $2n\geq \mu^j_1\geq\dots\geq\mu^j_r\geq 0, j\in [s]$. We fix
 $\mu$ in this section once and for all.

Make the definition:
$$\mathcal{H}_{\mu}(\mf,\mg)=\{\phi\in \home(M,V): \phi(F^j_a)\subset
G^j_{2n-\mu^j_a}, a\in [r],j\in [s]\}.$$
(In the next section, we will need to use $\mathcal{H}_{\mu}(\mf,\mg)$ in the case when
$G^j_{\bull}$ is an arbitrary complete flag on $V$, not necessarily an isotropic flag.)

 Now assume that $(\mf,\mg)$
 is a generic point of $\Fl(M)^s\times \IFl(V)^s$, where
 $\IFl(V)$ is the full isotropic flag variety of $V$ and  $\Fl(M)$ is the
full flag variety of $M$. Then,
 \begin{theorem}\label{key}
 The following are equivalent:
 \begin{enumerate}
\item[(A)] $\mathcal{H}_{\mu}(\mf,\mg)$
is of the expected dimension $2nr-\sum_{j=1}^s |\mu^j|$, where
 $|\mu^j|:=\sum_{a=1}^r\mu^j_a$.
\item[(B)] For any $1\leq d\leq r$  and subsets ${\sm}^1,\dots,{\sm}^s$ of $[r]$
each of cardinality $d$ such that the product $\prod_{j=1}^s
[\OMC_{{\sm}^j}]\neq 0\in H^*(\Gr(d,r))$, the inequality
$\sum_{j=1}^s\sum_{a\in {\sm}^j} \mu^j_a \leq 2dn$ holds.
\end{enumerate}
\end{theorem}

We record the following corollary.
\begin{corollary}\label{key2}
Suppose that $2nr-\sum_{j=1}^s |\mu^j|=0$. Then (A) (or (B)) is equivalent to
the condition
$$\bigl(V_{\mu^1}\tensor V_{\mu^2}\tensor\cdots\tensor
V_{\mu^s}\bigr)^{SL(r)}\neq 0,$$
where $V_{\mu^j}$ is the irreducible representation of $\SL(r)$ as in Section
 6.2.
\end{corollary}
\begin{proof}
The equivalence of the condition in the corollary and condition (B)
in Theorem ~\ref{key} is a consequence of the Knutson-Tao saturation
theorem for $\SL(n)$ [KT], together with Klyachko's work [K]: these  works
together
characterize the existence of invariants in a tensor product of
$\SL(r)$ representations by a system of inequalities, which is (B)
 (cf. [F$_4$]).
\end{proof}
\subsection{(A) implies (B) in Theorem ~\ref{key}}
This follows from ~\cite{gh}, Proposition 2.8 (1). The idea is that if
$(B)$ fails, pick an $S\in\cap_{j=1}^s
\OMM_{{\sm}^j}(F^j_{\bull})\subset\Gr(d,M)$, compute the expected
dimension of the vector space $\{\phi\in \mathcal{H}_{\mu}(\mf,\mg):
\phi(S)=0\}$, and find it to be greater than the expected dimension
of  $\mathcal{H}_{\mu}(\mf,\mg)$.
\subsection{(B) implies (A) in Theorem ~\ref{key}} We proceed by induction on $r$.  To show that (B)
implies (A), suppose $\phi^o$ is a general member of
$\mathcal{H}_{\mu}(\mf,\mg)$, $S=\ker(\phi^o)$ and assume
$S\in\cap_{j=1}^s \OMM_{{\sm}^j}(F^j_{\bull})\subset\Gr(d,M)$, for $B^j=
\{b^j_1< \dots <b^j_d\}$.

We will use ideas of  Schofield ~\cite{sch} in the proof. We replace
the Ext groups in ~\cite{sch} by the  cohomology
 of suitable (2-step) complexes, and replace the long exact sequences
in cohomology by the snake lemma. There are two other ingredients
required (beyond the technique of Schofield).
\begin{itemize}
\item A critical dimension count (Proposition ~\ref{count}), for which
we need Theorem ~\ref{wilson1}.
\item An idea from ~\cite{gh} on genericity of induced structures
(we could have used another technique of Schofield instead as well).
\end{itemize}
We were unable to use the strategy of ~\cite{gh}, Section 5
(essentially because we could not show that the induced flags on $S$
and $V/S$ were ``mutually generic'' in the situation here).

\subsection{The set up, and the key inequality}
Let $L=\im(\phi^o)$. Let $\mathcal{A}^0(\home(M,V))=\home(M,V)$, and
$$\mathcal{A}^1(\home(M,V))=\oplus_{j=1}^s \home(M,V)/P^j,$$
where
$$P^j= \{\tau\in \home(M,V): \tau(F^j_a)\subset G^j_{2n-\mu^j_a}, a=1,\dots,r\}.$$
There is a natural differential (a direct sum of projections)
$d:\mathcal{A}^0(\home(M,V))\to\mathcal{A}^1(\home(M,V)).$ Call this
(2-term) complex $\mathcal{A}^{*}(\home(M,V))$. Clearly,
$H^0(\ma^{*}(\home(M,V)))$ is the same as $\mathcal{H}_{\mu}(M,V)$
and $\chi(\ma^{*}(\home(M,V)))$ is the ``expected dimension'' of
$\mathcal{H}_{\mu}(M,V)$, which is  $2nr-\sum_{j=1}^s |\mu^j|.$ To
make the dependence on $(\mf,\mg)$ and $\mu$ precise, we sometimes
write this complex as $\ma^{*}(M,V,\mf,\mg,\mu)$.

Hypothesis (A) is clearly equivalent to
\begin{enumerate}
\item[(C)] $h^1(\ma^{*}(\home(M,V)))=0$.
\end{enumerate}
We have surjections $\home(M,V)\leto{\tau}\home(S,V)\leto{\pi}
\home(S,V/L)$. Let
$$\mathcal{A}^0(\home(S,V))=\home(S,V),\
\mathcal{A}^0(\home(S,V/L))=\home(S,V/L)$$ and
$$\mathcal{A}^1(\home(S,V))=\oplus_{j=1}^s \home(S,V)/\tau(P^j),\
\mathcal{A}^1(\home(S,V/L))=\home(S,V/L)/\pi\circ \tau(P^j).$$

There are natural differentials $d:\ma^0\to \ma^1$ in each case. It
is easy to see that
\begin{itemize}
\item $$\tau(P^j)=\{\phi\in\home(S,V): \phi(F^j(S)_x)\subset
G^j_{2n-\mu^j_{{\ssm}^j_x}}, x=1,\dots,d\}.$$ So, if we let
$\gamma^j_x=\mu^j_{{\ssm}^j_x}$,  the (2-term) complex
$\ma^{*}{(\home(S,V))}$ is the same as
$\ma^{*}(S,V,\mf(S),\mg,\gamma)$ where $\mf(S)$ is the induced
$s$-tuple of flags on $S$.
\item $$\pi\circ \tau(P^j)=\{\phi\in\home(S,V/L):
\phi(F^j(S)_x)\subset G^j_{2n-\gamma^j_x}+L/L, x=1,\dots,d\}.$$
\end{itemize}
Define numbers $\delta^j_x$ by
$$2n-\dim(L)-\delta^j_x =2n-\gamma^j_x -  \dim (G^j_{2n-\gamma^j_x}\cap
L)$$ and hence,  $\ma^{*}(\home(S,V/L))=\ma^{*}(S,V/L, \mf(S),
\mg(V/L),\delta).$

The following critical result will be proved in Section 5.6 by a dimension
 count.
\begin{proposition}\label{count}
$$h^1(\ma^{*}(\home(M,V)))\leq -\chi(\ma^{*}(\home(S,V/L))).$$
\end{proposition}

\subsection{Proof of (B) implies (A) in Theorem ~\ref{key} assuming Proposition
~\ref{count}} The map of (2-term) complexes
$\ma^{*}(\home(M,V))\to\ma^{*}(\home(S,V/L))$ is surjective in each
degree. Therefore, the map
\begin{equation}\label{yellow}
\theta:H^1(\ma^{*}(\home(M,V)))\to H^1(\ma^{*}(\home(S,V/L)))
\end{equation}
is a surjection.

 Assume Proposition ~\ref{count}. Hence, \begin{equation}\label{ma}
h^1(\ma^{*}(\home(M,V)))\leq -\chi(\ma^{*}(\home(S,V/L)))\leq
h^1(\ma^{*}(\home(S,V/L))).
\end{equation}
The surjection ~\eqref{yellow} and the Inequality ~\eqref{ma} imply
that $\theta$ is an isomorphism. Thus,  equality holds in all the inequalities
in  ~\eqref{ma}; in particular, $h^0(\ma^{*}(\home(S,V/L)))=0$.

The isomorphism $\theta$ factors as
$$H^1(\ma^{*}(\home(M,V)))\leto{\bar{\tau}}H^1(\ma^{*}(\home(S,V)))\to
H^1(\ma^{*}(\home(S,V/L))),$$ where the map $\bar{\tau}$ is induced from
the surjective map of complexes $$\ma^{*}(\home(M,V))\to
\ma^{*}(\home(S,V)),$$ and is hence surjective. We conclude that
$\bar{\tau}$ is an isomorphism.

It follows from ~\cite{gh} that the induced flags $(\mf(S),\mg)$ can
also be assumed to be suitably generic (see Section ~\ref{recolle}).
Now, by induction on the dimension of $M$, assume the validity of Theorem ~\ref{key}
with $M$ replaced by
$S$ and $\mu^j$ replaced by $\gamma^j$. The hypothesis (B) holds
because of Lemma ~\ref{yellow1} below.

Therefore, by conclusion (C) (valid by induction), we find that
$h^1(\ma^{*}(\home(S,V)))=0$. Hence $h^1(\ma^{*}(\home(M,V)))=0$ as
desired. This completes the proof of Theorem ~\ref{key}.
\begin{lemma}\label{yellow1} Let $1\leq d'\leq d$.
In the above situation, suppose $C^1,\dots,C^s$ are subsets of $[d]$
each of cardinality $d'$ such that $\prod_{j=1}^s [\OMC_{C^j}]\neq
0\in H^*(\Gr(d',d))$. Then, if $T^j=\{{\ssm}^j_a: a\in C^j\}$,
\begin{enumerate}
\item[(i)] $\prod_{j=1}^s [\OMC_{T^j}]\neq 0\in H^*(\Gr(d',r))$.
\item[(ii)] $\sum_{j=1}^s\sum_{a\in C^j}\gamma^j_a =\sum_{j=1}^s \sum_{b\in
T^j} \mu^j_b\leq 2d'n$.
\end{enumerate}
\end{lemma}
\begin{proof}
Since $S\in\cap_{j=1}^s \OMM_{{\sm}^j}(F^j_{\bull})\subset\Gr(d,M)$,
and $F^j$ are generic flags on $M$, the product $\prod_{j=1}^s
[\OMC_{{\sm}^j}]\neq 0\in H^*(\Gr(d,r))$.  Now using the hypothesis
$\prod_{j=1}^s [\OMC_{C^j}]\neq 0\in H^*(\Gr(d',d))$, and
\cite{ful4}, Proposition 1, we conclude that (i) holds. Because of
our assumption (B), (i) implies the inequality in (ii). The equality
in (ii) is trivial.
     \end{proof}

\subsection{Genericity of induced structures}\label{recolle}
Let $U\subseteq \Fl(M)^s\times\IFl(V)^s$ be a nonempty open subset
of the space of pair of flags which is generic for our dimension
calculations. There is a nonempty open subset
$\hat{U}\subseteq\Fl(S)^s\times \IFl(V)^s$ of points $(\mh,\mg')$
for which $(\mh,\mg')$ is generic for the application of induction
on $\ma^{*}(S,V,\mh,\mg',\gamma)$. The projection of $\hat{U}$ to
$\IFl(V)^s$ is a (diagonal) $\operatorname{Sp}(V)$-invariant open
subset of $\IFl(V)^s$ which does not depend upon the choice of $S\in
\Gr(d,M)$.

We obtain a point $(\mf(S),\mg)\in \Fl(S)^s\times \IFl(V)^s$ as in
the previous section. Since  $\mg$ is generic, it may be assumed to
be in the projection of $\hat{U}$ to $\IFl(V)^s$. Let $\mh\in
\Fl(S)^s$ be close to $\mf(S)$ (in the complex topology) such that
$(\mh,\mg)\in \hat{U}$. Now find $\mf'\in \Fl(M)^s$ so that
$\mf'(S)=\mh$, $\mf'(M/S)=\mf(M/S)$ and $S\in \cap_{j=1}^s
\OMM_{{\sm}^j}({F'}^j_{\bull})$ (see Lemma 2.4 in ~\cite{gh}). Clearly
$(\mf',\mg)$ is close to $(\mf,\mg)$ and therefore is generic in
$\Fl(M)^s\times \Fl(V)^s$ as well (and hence belongs to $U$). Now $\phi^o\in
H^0(\ma^{*}(M,V,\mf',\mg,\mu))$ and we replace
  $(\mf,\mg)$ by $(\mf',\mg)$.

\subsection{Proof of Proposition ~\ref{count}}
Since we already know the Euler characteristic of
$\ma^*(\home(M,V))$, we want to give a formula for
$h^0(\ma^*(\home(M,V)))=\mh_{\mu}(\mf,\mg)$.

 If we drop the isotropy condition on the flags $\mg$, then we
are in a situation where Schofield's original argument can be
applied.  We will
compute $\mh_{\mu}(\mf,\mg)$ by a procedure that ``essentially" does
not see the isotropy condition on  the flags $\mg$, and hence
Schofield's set up generalizes.

To calculate the dimension of $\mh_{\mu}(\mf,\mg)$, we can fiber the
universal parameter space of triples $(\phi,\mf,\mg)$ with $\phi$
generic in $\mh_{\mu}(\mf,\mg)$, over the parameter space of pairs
$(\im(\phi),\mg)$ so that $\im(\phi)$ is in its generic Schubert
state with respect to the flag $\mg$. The dimension of the choices
of such pairs is the expected one, thanks to our main transversality
result Theorem ~\ref{wilson1}.
\subsection{Calculation of $\chi(\ma^*(\home(S,V/L)))$}
  First, as in Subsection 5.3,  $\pi\circ\tau(P^j)$ is the same as
$$\{\theta\in\home(S,V/L): \theta(F^j(S)_x)\subset
G^j_{2n-\gamma^j_{x}}+L/L,\ x=1,\dots,d\},$$ which has dimension
$\sum_{x=1}^d (2n -\gamma^j_{x}- \dim (L\cap G^j_{2n-\gamma^j_x})).$
Therefore, $\chi(\ma^*(\home(S,V/L)))$ equals
$$ d(2n-r+d)-\sum_{j=1}^s \bigl(d(2n-r+d)-\sum_{x=1}^d (2n -\gamma^j_x- \dim (L\cap
G^j_{2n-\gamma^j_x}))\bigr)$$
\begin{equation}\label{gloss}
=d(2n-r+d)+\sum_{j=1}^s\sum_{x=1}^d\bigl(r-d -\gamma^j_{x}- \dim
(L\cap G^j_{2n-\gamma^j_{x}})\bigr) \end{equation}

\subsection{The dimension of $\mh_{\mu}(\mf,\mg)$:} Assume that
$$L\in \cap_{j=1}^s \OMM_{A^j}(G^j_{\bull})\subset
\Gr(r-d,V).$$  Let $U_{A}(V,s)\subset
\IFl(V)^s$ be the open subset consisting of isotropic flags
$E^1_{\bull}, \dots, E^s_{\bull}$ such that the intersection
$\cap_{j=1}^s \OMM_{A^j}(E^j_{\bull})$ inside $\Gr(r-d, V)$ is proper. By Theorem
~\ref{wilson1}, this is nonempty. Let $\mathcal{U}$ denote the parameter space
\begin{equation}\notag
\begin{split}\{(\mf,\mg,\phi, S,L):\ & \mf\in \Fl(M)^s, \mg\in U_A(V,s), L\in
\cap_{j=1}^s \OMM_{A^j}(G^j_{\bull}), S\in \cap_{j=1}^s
\OMM_{B^j}(F^j_{\bull}),\\ & \phi\in \mh_{\mu}(\mf,\mg),
S=\ker(\phi), L=\im(\phi)\}.
\end{split}
\end{equation}
Now, the dimension of $\mathcal{H}_{\mu}(\mf,\mg)$ equals the
dimension of a generic fiber of  $\mathcal{U}\to \FL(M)^s\times
U_A(V,s)$. We will therefore need to give an upper bound for the
dimension of irreducible components of $\mathcal{U}$.

Clearly, $\mathcal{U}$ maps to the  parameter space
\begin{equation}\notag
\begin{split}
\mv=\{(\mg,\phi,S, L):\ & \mg\in U_A(V,s), L\in \cap_{j=1}^s
\OMM_{A^j}(G^j_{\bull}), S\in \Gr(d,M), \phi\in\home(M,V),\\
& S=\ker(\phi), L=\im(\phi) \}.
\end{split}
\end{equation}
\subsubsection{The dimension of $\mathcal{V}$}
Clearly $\mv$ fibers over the space
$$\mw=\{(L,\mg):\mg\in U_A(V,s), L\in \cap_{j=1}^s
\OMM_{A^j}(G^j_{\bull})\}$$ with irreducible fibers of dimension
$d(r-d) +(r-d)^2$ (the first term is the dimension of the space of
choices of $S$ and the second term is the dimension of the space of
isomorphisms $M/S\leto{\phi}L$).

Now the fiber dimension of $\mw\to U_A(V,s)$ is
$((r-d)(2n-r+d)-\sum_{j=1}^s \codim (\OMM_{A^j}))$. Therefore, the
dimension of any irreducible component of $\mv$ is no more than
$$\dim(\IFl(V)^s)  +\bigl((r-d)(2n-r+d)-\sum_{j=1}^s \codim (\OMM_{A^j})\bigr)
+ d(r-d) +(r-d)^2.$$

By simplifying the above expression, we find
\begin{equation}\label{name}
\dim \mathcal{V}\leq 2nr +(r-d -2n)d +\dim(\IFl(V)^s) - \sum_{j=1}^s
\codim (\OMM_{A^j}).
\end{equation}
\subsubsection{Codimension of $\OMM_{A^j}$}
\begin{lemma}\label{calculator} $\dim (\OMM_{A^j}) \leq \sum_{a\not\in {\sm}^j}(2n-\mu^j_a
- \dim (L\cap G^j_{2n-\mu^j_{a}}))$, and hence

$\codim(\OMM_{A^j}) \geq \sum_{a\not\in {\sm}^j}(\mu^j_a + \dim
(L\cap G^j_{2n-\mu^j_{a}})-(r-d)).$
\end{lemma}
\begin{proof}
Let $A^j= \{a_1<\cdots <a_{r-d}\}$. Therefore, $\dim (\OMM_{A^j})$
equals $\sum_{b=1}^{r-d}(a_{b}-b)$. Let $[r]\setminus{\sm}^j=
\{x_1<\cdots<x_{r-d}\}$. It suffices therefore to show that, for any
$1\leq b \leq r-d$,
$$a_b -b \leq 2n-\mu^j_{x_b} - \dim (L\cap G^j_{2n-\mu^j_{x_b}}),$$
i.e.,
\begin{equation}\label{50}
a_b \leq b+ 2n-\mu^j_{x_b} - \dim (L\cap G^j_{2n-\mu^j_{x_b}}).
\end{equation}
Set $\dim (L\cap G^j_{2n-\mu^j_{x_b}})=b + y$, where $y\geq 0$ (since
$L\cap G^j_{2n-\mu^j_{x_b}}$ contains $\phi(F^j_{x_b})$ which is 
$b$-dimensional).

Hence, to prove Equation (\ref{50}), we need to show that $L\cap G^j_{2n-\mu^j_{x_b} -y}$ is at
least $b$-dimensional, which is now immediate because $L\cap
G^j_{2n-\mu^j_{x_b}}$ is $(b+y)$-dimensional.
\end{proof}
\subsubsection{The fiber dimension of $\mathcal{U}\to\mathcal{V}$} We
need to do the following basic calculation. Let $S\subset M$ be a
$d$-dimensional subspace, $L$ a subspace of $V$ of dimension $r-d$
and $B\in S(d,r)$.
\begin{lemma}\label{exam} Let
$G_{\bull}$ be a fixed flag on $V$, and $\phi:M\to V$ a map with
kernel $S$ and image $L\in \OMM_{A}(G_{\bull})$. The dimension of
the space of flags $F{\bull}\in \Fl(M)$ so that $S\in
\OMM_{{\sm}}(F_{\bull})$ and $\phi(F_a)\subset G_{2n-\mu_a}$ for
$a=1,\dots,r$, equals
$$\dim(\Fl(M))-\bigl(\codim (\OMM_{{\sm}}) +\sum_{a\not\in {\sm}}
(r-d-\dim(L\cap G_{2n-\mu_a}))\bigr).$$
\end{lemma}
\begin{proof}
There are two sets of conditions imposed on $F_{\bull}$. The first
one is that $S\in \OMM_{{\sm}}(F_{\bull})$; the second is that
$\phi(F_a)\subset G_{2n-\mu_a}\cap L$ for $a\not\in {\sm}$. These
conditions are independent.
\end{proof}
Therefore the fiber dimension of $\mathcal{U}\to \mathcal{V}$ equals
\begin{equation}\label{newton}
\dim(\Fl(M)^s)- \sum_{j=1}^s\bigl(\codim(\OMM_{{\sm}^j})
+\sum_{a\not\in {\sm}^j}(r-d-\dim(L\cap G^j_{2n-\mu^j_a}))\bigr).
\end{equation}
\subsection{Conclusion of the proof of Proposition ~\ref{count}}
The dimension of any irreducible component of $\mathcal{U}$ is less
than or equal to the sum of the fiber dimension of $\mathcal{U}\to
\mathcal{V}$ (given by ~\eqref{newton}) and the right hand side of
Inequality ~\eqref{name}. This sum equals
\begin{equation}\label{rememberr}
\begin{split}
& \dim(\Fl(M)^s)- \sum_{j=1}^s\bigl(\codim(\OMM_{{\sm}^j})
+\sum_{a\not\in {\sm}^j}(r-d-\dim(L\cap G^j_{2n-\mu^j_a}))\bigr)\\
& + 2nr +(r-d -2n)d +\dim(\IFl(V)^s) - \sum_{j=1}^s \codim
(\OMM_{A^j}).
\end{split}
\end{equation}
The dimension of $\mh_{\mu}(\mf,\mg)$, which equals the dimension of
the generic fiber of the map
 $\mathcal{U}\to \Fl(M)^s\times U_A(V,s)$,  is therefore less
than or equal to the integer ~\eqref{rememberr} minus  the dimension
of  $\Fl(M)^s \times U_A(V,s)$. Therefore, using Lemma
~\ref{calculator},
\begin{equation}\notag
\begin{split}
\dim (\mh_{\mu}(\mf,\mg))\leq &-\sum_{j=1}^s\bigl(\codim
(\OMM_{B^j}) + \sum_{a\not\in {\sm}^j}(r-d -\dim(L\cap
G^j_{2n-\mu^j_a}))\bigr)\\
& +2nr + d(r-d-2n) -\sum_{j=1}^s\sum_{a\not\in {\sm}^j}
(\mu^j_a+\dim
(L\cap G^j_{2n-\mu^j_a})-(r-d) )\\
&= 2nr + d(r-d-2n) -\sum_{j=1}^s \codim(\OMM_{{\sm}^j})
-\sum_{j=1}^s\sum_{a\not\in {\sm}^j}\mu^j_a.
\end{split}
\end{equation}

Now, $h^1(\ma^*(M,V))$ equals the dimension of $\mh_{\mu}(\mf,\mg)$
minus the Euler characteristic which is
$2nr-\sum_{j=1}^s\sum_{a=1}^r\mu^j_a$. This yields that
$h^1(\ma^*(M,V))$ is less than or equal to
\begin{equation}\label{bertrand}
d(r-d-2n) -\sum_{j=1}^s\codim(\OMM_{{\sm}^j})
+\sum_{j=1}^s\sum_{a\in {\sm}^j}\mu^j_a,
\end{equation}
which can be written in an illuminating (and expected) form
$$(\dim\Gr(d,r)-\sum_j\codim(\OMM_{{\sm}^j})) + (\sum_{j=1}^s\sum_{a\in {\sm}^j}\mu^j_a )-2nd.$$

 Comparing Equation ~\eqref{gloss} with  Equation ~\eqref{bertrand}, it 
suffices to show (for each $j$)
$$\sum_{x=1}^d (d-r+{\ssm}^j_x-x) +\sum_{a\in {\sm}^j}\mu^j_a\leq
-\sum_{x=1}^d((r-d) -\gamma^j_x- \dim (L\cap G^j_{2n-\gamma^j_x})),$$
which rearranges to 
$$\sum_{x=1}^d \dim (L\cap G^j_{2n-\gamma^j_x})\geq \sum_{x=1}^d
 ({\ssm}^j_x-x).$$
But, this  is clear because $L\cap G^j_{2n-\gamma^j_x}$ contains the
${\ssm}^j_x-x$ dimensional subspace $\phi(F^j_{{\ssm}^j_x})$. This completes
 the proof of Proposition ~\ref{count} and hence Theorem ~\ref{key} is proved.
\qed

\section{Tensor product decomposition for $\SL(r)$ versus $\Sp(2n)$ ~\label{clef}}

Let $G,B$ be as in the beginning of Section 2.
\subsection{Line bundles on $G/B$} For a character $\chi:B\to \Bbb{C}^*$, let
 $\ml_{\chi}$ be the corresponding  $G$-equivariant (homogeneous)
 line bundle on $G/B$ associated to the principal $B$-bundle $G\to G/B$ via
 the character $\chi^{-1}$. Recall
 that its total space is  the  set of all
pairs $\{(g,c):g\in G, c\in \Bbb C\}$ modulo the equivalence
$$(g,c)\sim (gb,\chi(b)c).$$
 If $\chi$
is a dominant weight, then
$$H^0(G/B,\ml_{\chi})= V_{\chi}^*,$$
where (as in Section 2) $V_{\chi}$ is the irreducible representation of $G$ with
highest weight $\chi$.

 If $f:G'\to G$ is a map of algebraic groups so
that $f(B')\subset B$, where $B'$ and $B$ are Borel subgroups of $G'$
and $G$. Then, $f$ induces a map $\bar{f}:G'/B'\to G/B$. It can be checked
that $f^*\ml_{\chi}= \ml_{\chi\circ f}$, where $\chi\circ f $ is the
composition of $\chi$ with $f_{|B'}: B'\to B$.

\vskip1ex

{\bf Specific identification in the $\SL(r)$ case:} Let $B$
be the standard Borel subgroup of $\SL(r)$ consisting of upper triangular
matrices as in Subsection 2.1. Then, the full flag variety $\Fl(r)=\SL(r)/B$  can be
identified with the space of full flags $F_{\bull}$ on a $r$-dimensional
vector space $M$.  As such it receives natural line bundles
$\ml_a=\det(F_a)^{*}$ ($1\leq a \leq r$). This line bundle
is isomorphic with the line bundle  $\ml_{\omega_a}$, where $\omega_a$ is the
character of $B$
which takes the diagonal matrix  $(t_1, \dots, t_r)$ to the product
$\prod_{i=1}^a t_i$. Recall that  ${\omega_a}$
is the $a^{\text{th}}$
fundamental weight.
\subsection{Representations of $\GL(r)$} Irreducible
polynomial representations of the general linear group $\GL(r)$ correspond to the sequences of
 integers $\mu= (\mu_1\geq \cdots\geq\mu_r\geq 0)$  (also called Young diagrams or partitions).
 Given a sequence of integers
$\mu=(\mu_1\geq\mu_2\geq\cdots\geq\mu_r\geq 0)$, form the
line bundle
$$\ml_{\mu}=\prod_{a=1}^r\ml_a^{\mu_{a}-\mu_{a+1}},$$
where  ($\mu_{r+1}=0$).
Then,  $H^0(\Fl(r),\ml_{\mu})= V_{\mu}^*$, where
$V_{\mu}$ is the irreducible representation of $\GL(r)$ with highest weight
$\mu$. To any such $\mu$, we associate the character  $\bar{\mu}$  of $B$ by
\[\bar{\mu}(t)=\prod_{i=1}^r\,t_i^{\mu_i},\,\,\text{where}\,\,t=\text{diag}(t_1, \dots, t_r).\]
Then, clearly, $\ml_{\mu}\simeq \ml_{\bar{\mu}}$.

\subsection{A Basic Calculation} As earlier in Subsection 2.2, 
let $V=\Bbb C^{2n}$
 be equipped with the symplectic form $\langle\,,\,\rangle$
and $M$ a vector space of dimension $r$.

We will now make a basic determinantal calculation (which appeared
originally in a weaker form in ~\cite{invariant}, Lemma 4.5). Let
$P:=\home(M,V)$. Fix an integer $ m \leq 2n$.  For a full flag
$F_{\bull}$ on $M$, a full flag $G_{\bull}$ on $V$, and a partition
$\mu= (m\geq \mu_1\geq\cdots \geq \mu_r\geq 0)$, let
$P_m(\mu,F_{\bull},G_{\bull}$) denote the subspace
$$\{\phi\in \home(M,V): \phi(F_a)\subset
G_{m-\mu_a}, \,\forall a \in[r]\}.$$ Consider the vector bundle
$T_m(\mu)$ (respectively, $P_m(\mu)$)  on
$\Fl(M)\times \Fl(V)$ whose fiber over $(F_{\bull}, G_{\bull})$ is
$T_m(\mu,F_{\bull}, G_{\bull}) := \home(M,V)/P_m(\mu,F_{\bull},
G_{\bull})$ (respectively, $P_m(\mu,F_{\bull},G_{\bull}$)). We now calculate the determinant line bundles of
$P_m(\mu )$ and $T_m(\mu )$. We
begin the computation by setting (for $1\leq b\leq r$)
$$P_m^b= \{\phi\in \home(M,V): \phi(F_a)\subset
G_{m-\mu_a}, a=1,\dots,b\}.$$ There exist exact sequences for  $0\leq b
\leq r-1$:
$$0\to P_m^{b+1}\to P_m^b\to \home(F_{b+1}/F_b, V/G_{m-\mu_{b+1}})\to 0,$$
where $P_m^0:=\home(M,V).$
This permits us to write ($\ml_a:=\det(F_a)^{*}, \mk_b:=\det(G_b)^*)$
$$\det T_m(\mu )=\Bigl(\bigl(\prod_{a=1}^r \ml^{\mu_{a}-\mu_{a+1}}_a\bigr)\otimes
(\det M^*)^{2n-m}\Bigr)
\boxtimes \bigl((\det V)^r \tensor\prod_{a=1}^r
\mk_{m-\mu_a}\bigr), $$ where $\mu_{r+1}:=0$. (Here we have used the fact that
for vector spaces $V_1,V_2, \det(\home(V_1,V_2))\simeq ((\det V_1)^*)^{\dim V_2}
\otimes (\det V_2)^{\dim V_1}.$)
Let $\lambda$ be the partition conjugate to the partition $\lambda':=(m-\mu_r\geq \cdots\geq m-\mu_1\geq 0)$. We call $\lambda$ to be obtained from $\mu$ by {\it $m$-flip}.
 Then, we may rewrite
the above formula as
\begin{equation}\label{21}\det T_m(\mu )= \Bigl(\ml_{\mu} \otimes
 (\det M^*)^{2n-m}\Bigr)\boxtimes \Bigl((\det V)^r \tensor\ml_{\lambda}\Bigr).
\end{equation}

\subsection{The theta section}\label{determinant} Let $M$  and $V$
be as in Section 6.3. We take $m=2n$. Let $\{\mu^j\}_{1\leq j\leq s}$ be partitions:
$$2n\geq \mu^j_1\geq\cdots \geq \mu^j_r\geq 0$$
satisfying
 $$\sum_{j=1}^s |\mu^j| = 2nr,$$
where $ |\mu^j|:= \sum_{i=1}^r\,\mu_i^j$.

 Consider two $\Gl(M)\times \GL(V)$-equivariant bundles
$\mathcal{P}$ and $\mathcal{Q}$ on $\Fl(M)^s\times \Fl(V)^s$ (with the diagonal actions of $\GL(M)$ and $\GL(V)$ on $\Fl(M)^s$ and $\Fl(V)^s$ respectively).
 The
fiber of $\mathcal{P}$ over $(\mathcal{F},\mathcal{G})$ is just the
constant vector space $\home(M,V)$ (with the natural action of the
group $\Gl(M)\times \GL(V)$), where $\mathcal{F} = (F^1_{\bull}, \dots,
F^s_{\bull})$ is an $s$-tuple of full flags on $M$ and similarly for $ \mathcal{G}$.  The fiber of $\mathcal{Q}$ over
$(\mathcal{F},\mathcal{G})$ is the direct sum
$$\oplus_{j=1}^s T(\mu^j, F^j_{\bull},  G^j_{\bull}),$$
where $ T(\mu^j, F^j_{\bull},  G^j_{\bull}):= T_{2n}(\mu^j, F^j_{\bull},  G^j_{\bull}).$
Consider the canonical $\Gl(M)\times \GL(V)$-equivariant map
$S:\mathcal{P}\to \mathcal{Q}$ between
vector bundles (of the same rank) on $\Fl(M)^s\times \Fl(V)^s$
 induced via the quotient maps.
Thus, taking its determinant, we find  a  $\Gl(M)\times \GL(V)$-equivariant
 global section
 $\Theta\in \det(\mq)\tensor (\det{\mathcal{P}})^{-1}$ such that  $\Theta$ vanishes at
 $(\mf,\mg)$ if and only if $\mh_{\mu}(\mf,\mg)\neq 0$, where
 $\mh_{\mu}(\mf,\mg)$ is as defined in the beginning of Section 5
(extended for any, not necessarily isotropic, complete flags $\mathcal{G}=
(G^1_{\bull}, \dots, G^s_{\bull})$).
By Equation (\ref{21}),
$$\det(\mq)=\mathcal{L}_{\mu^1}\boxtimes\cdots\boxtimes
\mathcal{L}_{\mu^s}\boxtimes \bigl(\det(V)^r \otimes
\mathcal{L}_{\lambda^1}\bigr)\boxtimes \cdots \boxtimes \bigl(\det(V)^r\otimes
\mathcal{L}_{\lambda^s}\bigr).
$$
Here $\lambda^i$ is obtained from $\mu^i$ by $2n$-flip.

Fixing a point $\mf\in \Fl(M)^s$, on restriction,  we get (as in
~\cite{invariant}) a $\SL(V)$-invariant section $\Theta_{\mf}$ in
$$H^0(\Fl(V)^s,\mathcal{L}_{\lambda^1}\boxtimes\cdots\boxtimes\mathcal{L}_{\lambda^s})^{\SL(V)}.$$
Restricting this to $\IFl(V)^s$ we find a section  in
$$H^0(\IFl(V)^s,\mathcal{L}_{\lambda_C^1}\boxtimes\cdots\boxtimes
\mathcal{L}_{\lambda_C^s})^{\Sp(V)}=\bigl(M^*_{\lambda_C^1}\tensor
\cdots\tensor M^*_{\lambda_C^s}\bigr)^{\Sp(2n)},$$
where $\IFl(V)$ is the space of (full) isotropic flags on $V$, $\lambda_C^i$ is the weight of $\Sp(2n)$ obtained by the restriction of ${\lambda}^i $
to the Cartan of $\Sp(2n)$ and $M_\nu$ is the irreducible representation
of $\Sp(2n)$ with highest weight $\nu$. (Observe that for a dominant weight
$\lambda$ of $\SL(2n)$, its restriction $\lambda_C$ is dominant for $\Sp(2n)$.)

We can therefore record the following consequence of Corollary
~\ref{key2}.
\begin{corollary}\label{walk} Let $\{\mu^j\}_{1\leq j\leq s}$ be partitions:
$(2n\geq \mu^j_1\geq\cdots \geq \mu^j_r\geq 0)$
such that
 $\sum_j |\mu^j| = 2nr.$ Assume further that
 $$\bigl(V_{\mu^1}\tensor \cdots\tensor
V_{\mu^s}\bigr)^{SL(r)}\neq 0.$$ Then,
$$\bigl(M_{\nu^1}\tensor
\cdots\tensor M_{\nu^s}\bigr)^{\Sp(2n)}\neq 0,$$
where $\nu^j:=\lambda_C^j$ and $\lambda^j$ is obtained from $\mu^j$ by $2n$-flip.
\end{corollary}

We now come to the main theorem of this section.
\begin{theorem}\label{clef'1}
Let $V_{\lambda^1},\dots,V_{\lambda^s}$ be irreducible representations of
$\SL(2n)$ (with highest weights $\lambda^1, \dots, \lambda^s$ respectively)
 such that their tensor product has a nonzero $\SL(2n)$-invariant. Then, the tensor
 product of the representations of $\Sp(2n)$ with
highest weights  $\lambda_C^1,\dots,\lambda_C^s$ has a
nonzero $\Sp(2n)$-invariant, where $\lambda_C^j$ is the restriction of
$\lambda^j$ to the Cartan $\frh^C$ of $\Sp(2n)$.
\end{theorem}
\begin{proof}
We would like to find some $r$ and partitions $\mu^1,\dots,\mu^{s}$ such that
$$2n\geq \mu^j_1\geq\cdots \geq \mu^j_r\geq 0$$
with
 $\sum_{j=1}^{s} |\mu^j| = 2nr,$
and such that $\lambda^j$ is obtained from $\mu^j$ by $2n$-flip.

To achieve this,   express ${\lambda}^j$ as  sequences of integers
(with $r$ large)
$$r\geq\lambda^j_1\geq\cdots \geq \lambda^j_{2n}\geq 0$$
 so that  (since the tensor product of the corresponding representations of $\SL(2n)$ has a nonzero $\SL(2n)$-invariant)
$$\sum_{j=1}^s\sum_{a=1}^{2n}\lambda^j_a=2nb' \,\,\text{for some integer}\, \, b'\geq 0,
$$
i.e.,
$$\sum_{j=1}^s\sum_{a=1}^{2n}(r-\lambda^j_a)=2nr +2nb,$$
where $b:=r(s-1)-b'$. Taking $r\geq \frac{s}{s-1} \lambda^j_1$ (for every $j$), we can assume that
$b\geq 0$.
 Rewrite the above equality as $\sum_{j=1}^s\sum_{a=1}^{2n}(r+b-(b+\lambda^j_a))=2n(r
 +b)$. So, we replace $r$ by $r+b$ and $\lambda^j_a$ by $b+\lambda^j_a$
 to assume $b=0$. Let $\mu^j$ be obtained from $\lambda^j$ by $r$-flip. Since, by assumption,
 $$\bigl(V_{\lambda^1}\tensor\cdots\tensor
V_{\lambda^s}\bigr)^{\SL(2n)}\neq 0,$$ we obtain (using ordinary
duality)
$$\bigl(V_{{\lambda^1}'}\tensor\cdots\tensor
V_{{\lambda^s}'}\bigr)^{\SL(2n)}\neq 0,$$ where
${\lambda^j}'=(r-\lambda^j_{2n},\dots, r-\lambda^j_1)$.

Now using Grassmann duality (see Section ~\ref{gdual}),
$$\bigl(V_{\mu^1}\tensor\cdots\tensor
V_{\mu^s}\bigr)^{SL(r)}\neq 0.$$ Now, apply Corollary ~\ref{walk}.
(Observe that $2n$-flip of $\mu^j$ is the partition $\lambda_1^j-\lambda_{2n}^j
\geq \cdots \geq \lambda_{2n-1}^j-\lambda_{2n}^j\geq 0$.)
This completes the proof of the theorem.
\end{proof}
\subsection{Grassmann duality}\label{gdual} Let $r$ and $k$ be positive integers and
let $m=r+k$. Given a Young diagram $\mu=(k\geq \mu_1\geq\dots\geq\mu_r\geq 0)$,
one obtains naturally a cohomology
class $\omega_{\mu}$ of $\Gr(r,m)$, which is  the cycle class of
$\bar{\Omega}_{A(\mu)}$ with $A(\mu)=\{m-r+a-\mu_a|a\in [r]\}$.
It is  known that if $\mu^1,\dots,\mu^s$ are such Young diagrams
with $\sum_{j=1}^{s} |\mu^j| = kr,$ then the dimension of $\bigl(V_{\mu^1}\tensor\cdots\tensor
V_{\mu^s}\bigr)^{SL(r)}$  equals the coefficient of the cycle class
of a point in the cup product $\prod_{j=1}^s\omega_{\mu^j}\in
H^{2rk}(\Gr(r,m))$ (see, e.g., ~\cite{fulton2}, Chapter 9).

Clearly, under the natural duality map
$\Gr(r,\Bbb{C}^m)=\Gr(k,(\Bbb{C}^m)^*)$, the  class
$\omega_{\mu}$ goes to $\omega_{\tilde{\mu}}$, where $\tilde{\mu}$ is the
conjugate diagram to $\mu$ (see, e.g., ~\cite{fulton2}, Exercise 20,
page 152). We therefore obtain the following consequence, which we
call the {\it Grassmann duality}.
\begin{lemma}
Suppose  $\sum_{j=1}^{s} |\mu^j| = kr.$ Then,
$$\dim \bigl(V_{\mu^1}\tensor\cdots\tensor V_{\mu^s}\bigr)^{SL(r)}=\dim \bigl(V_{\tilde{\mu}^1}\tensor\cdots\tensor
V_{\tilde{\mu}^s}\bigr)^{SL(k)}.$$
\end{lemma}

We clearly have another ``ordinary duality''. Let $\nu^j$ be the dual
partition to $\mu^j$ (so that $\nu^j=(k-\mu^j_r,\dots,
k-\mu^j_1)$). Then, clearly,  $V_{\nu^j}$ is the $\SL(r)$-dual of
$V_{\mu^j}$ and hence

$$\dim \bigl(V_{\mu^1}\tensor\cdots\tensor V_{\mu^s}\bigr)^{SL(r)}=\dim \bigl(V_{{\nu}^1}\tensor\cdots\tensor
V_{{\nu}^s}\bigr)^{SL(r)}.$$
\section{Saturation for $\Sp(2n)$ and $\SO(2n+1)$}
Let $X(H^C)_+$ be the set of dominant characters of $H^C$
 (for the group $\Sp(2n)$) and similarly let  $X(H^B)_+$ be the set of dominant characters of $H^B$
 (for the group $\SO(2n+1)$). We have the following saturation theorems for the groups $\Sp(2n)$ and  $\SO(2n+1)$ respectively.

\begin{theorem}\label{saturation}
Given  $\nu^1,\dots,\nu^s \in X(H^C)_+$, the following
are equivalent:
\begin{enumerate}
\item For some positive integer $N$, the tensor product of $\Sp(2n)$-representations with highest weights
$N\nu^1, \dots,N\nu^s$ has a nonzero $\Sp(2n)$-invariant.
\item The
tensor product of representations  with highest weights $2\nu^1,
\dots,2\nu^s$ has a nonzero $\Sp(2n)$-invariant.
\end{enumerate}
\end{theorem}
\begin{proof}
For any semisimple algebra $\frg$ with Cartan subalgebra $\frh$, we
identify $\kappa : \frh \simto \frh^*$ via the normalized invariant
bilinear form $\langle\, ,\,\rangle$ on $\frh^*$, normalized so that
$\langle\theta ,\theta\rangle =2$ for the highest root $\theta$.
Under this identification, $\frh_+ \simto D$ and, moreover, $x_i\mapsto
2\omega_i/\langle\alpha_i ,\alpha_i\rangle$.

For any $s\geq 1$ and $G$ as in the beginning of Section 2, define
  \begin{align*}
\hat{\Gamma} (s,G) &= \bigl\{ (\lam^1, \dots ,\lam^s)\in \bigl(
X(H)_+\bigr)^s: \text{ for some }N\geq 1, \\
&\qquad\qquad \bigl( V_{N\lam^1}\otimes\cdots\otimes
V_{N\lam^s}\bigr)^G \neq 0\bigr\} .
  \end{align*}
Then, under the identification induced by $\kappa$ (and still
denoted by) $\kappa: \frh^s_+ \to D^s$, we have (cf., e.g., [Sj,
Theorem 7.6])
  \[
\kappa^{-1}\bigl( \hat{\Gamma}(s,G)\bigr) = \Gamma (s,K )\cap
\kappa^{-1}\bigl( (X(H)_+)^{s}\bigr) ,
  \]
where the eigencone $\Gamma (s,K)$ is defined in Section 4.

Recall from Subsection 2.2 the inclusion $i: \frh^C \hookrightarrow
\frh^{A_{2n-1}}$ and hence we get the dual (surjective) map $i^*:
\Bigl(\frh^{A_{2n-1}}\Bigr)^* \twoheadrightarrow (\frh^C)^*$.  We have
the following commutative diagram:
  \[   \tag{$*$}  \begin{CD}
\frh^C @>i>> \frh^{A_{2n-1}} \\
@VV{2\kappa_C}V @VV{\kappa_A}V \\
(\frh^C)^* @<i^*<< \bigl(\frh^{A_{2n-1}}\bigr)^* ,
  \end{CD}   \]
where $\kappa_C$ (respectively, $\kappa_A$) denotes $\kappa$ for $\Sp(2n)$
(respectively, $\SL(2n)$).

We only need to prove (1) $\Rightarrow$ (2).  By assumption, $\nu :=
(\nu^1, \cdots ,\nu^s) \in\hat{\Gamma}(s, \Sp(2n))$ and hence
$\kappa_C^{-1} \nu\in\Gamma (s, \Sp(2n))$.  But, clearly, $\Gamma (s,
\Sp(2n))\subset \Gamma (s, \SU(2n))$ and hence $\kappa_A\, i\,
\kappa_C^{-1}(\nu) \in \hat{\Gamma} (s,\, \SL(2n))$.  (Observe that
$\kappa_A\, i\, \kappa_C^{-1} (X(H^C)_+)$ $\subset Q^{A_{2n-1}}\cap
X(H^{A_{2n-1}})_+$, since $\kappa_C(\beta^\vee_j) =\beta_j$ for $1\leq
j<n$ and $\kappa_C(\beta^\vee_n) =\frac{1}{2}\, \beta_n$, where
$Q^{A_{2n-1}}$ is the root lattice of $\SL(2n)$.)  Moreover, since
$\kappa_A\, i\, \kappa_C^{-1}(\nu^j)\in Q^{A_{2n-1}}$, by the
saturation theorem of Knutson-Tao [KT],
  \[
\bigl( V_{\lam^1}\otimes\cdots\otimes V_{\lam^s}\bigr) ^{\SL(2n)}
\neq 0,
  \]
where $\lam^j := \kappa_A\, i\, \kappa_C^{-1}(\nu^j)$ and
$V_{\lam^j}$ is the irreducible representation of $\SL(2n)$ with
highest weight $\lam^j$.

Applying Theorem ~\ref{clef'1}, we get
  \[
\Bigl( M_{\lam_C^{1}} \otimes\cdots\otimes
M_{\lam_c^s}\Bigr)^{\Sp(2n)} \neq 0,
  \]
where $\lam_C^j := \lam^j_{|_{\frh^C}}$ and $M_{\lam^j_C}$ is the
irreducible representation of $\Sp(2n)$ with highest weight
$\lam_C^j$.  Moreover, by the commutativity of the diagram ($*$), we
get $\lam_C^j = 2\nu^j$.  This proves the theorem.
  \end{proof}

By using [KS, Section 4] and the above theorem, we get the corresponding
result for the odd orthogonal groups.
\begin{theorem}\label{saturation2}
Given  $\nu^1,\dots,\nu^s\in  X(H^B)_+$, the
following are equivalent:
\begin{enumerate}
\item For some positive integer $N$, the tensor product of $\SO(2n+1)$-representations  with highest weights
$N\nu^1, \dots,N\nu^s$ has a nonzero $\SO(2n+1)$-invariant.
\item The
tensor product of representations with highest weights $2\nu^1,
\dots,2\nu^s$ has a nonzero $\SO(2n+1)$-invariant.
\end{enumerate}
\end{theorem}
\begin{remark} (a) The general saturation result proved by Kapovich-Millson
 [KM] specialized to the case of $\Sp(2n)$ requires saturation factor of $4$
 as
opposed to our saturation factor of $2$ proved in Theorem ~\ref{saturation}. Our Theorem  ~\ref{saturation} is optimal for general weights
 $\nu^1,\dots,\nu^s \in X(H^C)_+$.

(b) We can prove an analogue of Theorem ~\ref{key} for $\SO(2n+1)$
by a similar method and use it to give a proof of Theorem
~\ref{saturation2}
 analogous to that of  Theorem  ~\ref{saturation}.
\end{remark}

As an immediate consequence of  Theorem ~\ref{saturation2}, we get the following special case of the general saturation result due to  Kapovich-Millson
 [KM].
\begin{corollary}\label{saturation3}
Let $\nu^1,\dots,\nu^s$ be dominant characters for
$\Spin(2n+1)$ (whose sum is not necessarily in the root lattice). Then, the
following are equivalent:
\begin{enumerate}
\item For some positive integer $N$, the tensor product of representations of
 $\Spin(2n+1)$ with highest weights
$N\nu^1, \dots,N\nu^s$ has a nonzero $\Spin(2n+1)$-invariant.
\item The
tensor product of representations with highest weights $4\nu^1,
\dots,4\nu^s$ has a nonzero $\Spin(2n+1)$-invariant.
\end{enumerate}
\end{corollary}

We would like to conjecture the following generalization of Theorems 
~\ref{wilson1}, ~\ref{newwilson2} and ~\ref{clef'1}.

Let $G$ be a connected simply-connected, semisimple complex
algebraic group and let $\sigma$ be a diagram automorphism of $G$
with fixed subgroup $K$.

  \begin{conjecture}\label{conj}  (a) For any standard parabolic subgroup $P$ and any
Bruhat cells $\Lambda^P_{w_1},\cdots ,\Lambda^P_{w_s}$ in $G/P$, there
exist elements $k_1,\cdots ,k_s\in K$ such that the intersection
$\bigcap^s_{i=1}k_i\Lambda^P_{w_i}$ is proper.

(b) Let $V_{\lam^1},\cdots ,V_{\lam^s}$ be irreducible
representations of $G$ (with highest weights $\lam^1,\cdots ,\lam^s$
respectively) such that their tensor product has a nonzero
$G$-invariant.  Then, the tensor product of irreducible $K$-modules
with highest weights $\lam^1_K,\cdots ,\lam^s_K$ has a nonzero
$K$-invariant, where $\lam^i_K$ is the restriction of $\lam^i$ to the Cartan
subalgebra $\frh_K := \frh^{\sigma}$ of $K$.

(Observe that $\lam^i_K$ is dominant for $K$ for any dominant weight
$\lam$ of $G$ with respect to the Borel subgroup $B^K := B^{\sigma}$
of $K$.)
  \end{conjecture}

  \section{Horn's problem  for symplectic groups}\label{july4}
  In this section we prove an analogue of Horn's problem for  symplectic groups.

Fix $r\leq n$ and let $I^1,\dots,I^s\in \FS(r,2n)$ be such that
\begin{equation}\label{postman}
\sum_{j=1}^s\codim({\LAM}_{{I^j}})=\dim \IG(r,2n).
\end{equation}

 Let $\beta_j$ be the
order preserving  bijection of $I^j\sqcup \bar{I}^j$ with $[2r]$ and
let $I_o^j$ be the subset $\beta(I^j)$ of $[2r]$ (of cardinality
$r$). For any $a\in [r]$ and $j \in [s]$, let
$\lambda^j_a:=|i_a^j\geq \tilde{I}^j|$ and $\mu^j_a:=2n-2r-\lambda^j_a$, where $I^j= \{i^j_1 <\cdots
<i^j_{r}\}$. Recall the definition of the deformed product $\odot_0$ in the cohomology $H^*(G/P)$
from [BK, Definition 18].
\begin{theorem}\label{challenge}
 Under the assumption  ~\eqref{postman}, the following conditions ($\alpha$) and
$(\beta)$ are equivalent:
\begin{enumerate}
\item[$(\alpha)$] $\prod_{j=1}^s[\bar{\LAM}_{{I^j}}]\neq 0\in H^{*}(\IG(r,2n), \odot_0).$
\item[$(\beta)$]
\begin{enumerate}
\item[$(\beta_1)$]
$\dim \IG(r,2r)= \sum_{j=1}^s \cosym^2(I^j).$
\item[$(\beta_2)$]
$\bigl(V_{\mu^1}\tensor\cdots\tensor
V_{\mu^s}\bigr)^{SL(r)}\neq 0$.
\item[$(\beta_3)$] $\prod_{j=1}^s[\bar{\LAM}_{{{I_o^j}}}]\neq 0\in
H^{*}(\IG(r,2r)).$
\end{enumerate}
\end{enumerate}
\end{theorem}

\begin{remark}\label{purbhoo}
(1) Assume that the condition $(\beta_1)$ is satisfied. Then, the  condition
 $(\beta_2)$ is equivalent to the following condition.

  For any $1\leq d\leq r$ and subsets ${\sm}^1,\dots,{\sm}^s$ of $[r]$ each of cardinality $d$
such that the product $\prod_{j=1}^s [\OMC_{{\sm}^j}]\neq 0\in
H^*(\Gr(d,r))$, the following inequality holds
$$\sum_{j=1}^s\sum_{a\in {\sm}^j} \mu^j_a \leq 2d(n-r).$$

This follows from  Corollary ~\ref{key2} applied to a space $M$ of dimension $r$ and $V=M^\perp/M$ of
dimension  $2n-2r$ and Identity (\ref{100}). In fact, we only need the equivalence of the assertion in  Corollary ~\ref{key2} and the $(B)$-part of Theorem 
~\ref{key}, which is due to Klyachko and Knutson-Tao (cf., the proof of 
 Corollary ~\ref{key2}).

(2) The
 condition ($\beta_2$) in the above theorem is recursive by
the work of Purbhoo and Sottile ~\cite{ps}.
The Purbhoo-Sottile recursion says that $(\beta_2)$ is equivalent to
a system of inequalities indexed by nonvanishing $s$-fold
intersections of Schubert cycle classes in the Grassmannians
$\Gr(d,r)$ where $0<d<r$.
\end{remark}

 \subsection{Proof of Theorem ~\ref{challenge}}
 We first observe that, under the assumption  ~\eqref{postman}, the following two conditions
 are equivalent.

\begin{enumerate}
\item[$(\gamma_1)$] $\prod_{j=1}^s[\bar{\LAM}_{{I^j}}]\neq 0\in
H^{*}(\IG(r,2n),\odot_0).$
\item[$(\gamma_2)$]  $\prod_{j=1}^s[\bar{\LAM}_{{I^j}}]\neq 0\in
H^{*}(\IG(r,2n))$ and the condition $(\beta_1)$ is satisfied.
\end{enumerate}

 To see this we use [BK, Theorem 15] and the following lemma.
\begin{lemma}
For $I\in \FS(r,2n)$, $\chi^C_{w_I}(x^C_r)=
\codim({{\LAM}}_{{I}})+\codim({\LAM}_{{{I_o}}}),$ where
 $\chi^C_{w_I}$ is defined just before  Lemma ~\ref{sosp}.
\end{lemma}
\begin{proof}
Consider the exact sequence ~\eqref{fun}. Let $L_M=
\operatorname{GL}(M)\times \Sp(M^{\perp}/M)$ be the Levi subgroup of
the stabilizer $P_M$ of $M$ in $\Sp(2n)$. The connected center $\Bbb{C}^*$ of $L_M$ (which is the
center of $\operatorname{GL}(M)$) acts as $t\mapsto t^{-1}$
(multiplication by $t^{-1}$) on $\home(M,{M}^{\perp}/M)$ and by
$t\mapsto t^{-2}$ on $\sym^2(M)^*$. We conclude the proof  by using
Proposition ~\ref{morning} and the formula for $\chi^C_{w_I}$ given
at the end of page 192 in ~\cite{BK}.
\end{proof}

Let $V$ be a vector space of dimension $2n$ with a nondegenerate
symplectic form and let $M\subset V$ be an isotropic subspace of dimension $r$.
Let $X=\IG(r,2n)$ and assume the condition $(\beta_1)$.
Let $E_{\bull}=(E^1_{\bull},\dots,E^s_{\bull})$ be an
$s$-tuple of isotropic flags on $V$ such that  for all $j\in [s],
M\in \LAM_{I^j}(E^j_{\bull})$.
Consider the exact sequence (as in Lemma ~\ref{gym}):
$$0\to \home(M,M^{\perp}/M)\leto{\xi} T(X)_M\leto{\phi} \Sym^2 M^*\to
0.$$
Let us abbreviate
$$\ma^0=\home(M,M^{\perp}/M),\  \mb^0=T(X)_M,\  \mc^0= \Sym^2 M^*.$$ For $j=1,\dots,s$, let
$$P^j:=\home(M,M^{\perp}/M)\cap (T(\LAM_{I^j}(E^j_{\bull}))_M),\
Q^j:=T(\LAM_{I^j}(E^j_{\bull}))_M,\
R^j=\phi(T(\LAM_{I^j}(E^j_{\bull}))_M).$$
Define
$$\ma^1=\oplus_{j=1}^s \ma^0/P^j,\  \mb^1=\oplus_{j=1}^s \mb^0/Q^j,\ \mc^1=\oplus_{j=1}^s
\mc^0/R^j.$$

There are natural differentials $$\ma^0\to\ma^1,\,\, \mb^0\to\mb^1,\,\,
\mc^0\to\mc^1,$$ and an exact sequence of two-term complexes
\begin{equation}\label{99}
0\to \ma^*\leto{\xi} \mb^*\leto{\phi} \mc^*\to
0.
\end{equation}
Our assumption ~\eqref{postman} implies that $\chi(\mb^*)=0$,
and the assumption $(\beta_1)$ (in view of Proposition ~\ref{morning} (3))
implies that $\chi(\mc^*)=0$. Therefore, from the above exact sequence,
$\chi(\ma^*)=0$ as well. Thus, by
 Proposition ~\ref{morning} (1)-(2),
\begin{equation}\label{100}
2(n-r)r-\sum_{j=1}^s|\mu^j|=0.\end{equation}
Moreover, $h^0(\mb^*)=0$ if and
only if $h^0(\ma^*)=0$ and $h^0(\mc^*)=0$ from the above exact sequence
 ~\eqref{99}.

We now prove:

\noindent
$(\alpha)\Rightarrow (\beta)$: Let $E_{\bull}=(E^1_{\bull},\dots,E^s_{\bull})$
be a generic
$s$-tuple of isotropic flags on $V$ such that  for all $j\in [s],
M\in \LAM_{I^j}(E^j_{\bull})$. This can be achieved by the virtue of
$(\alpha)$ and by simultaneously translating each $E^j_{\bull}$ by the same
element  of $\Sp(2n)$.
Since $E_{\bull}$ is generic, $\Omega = \cap_{j=1}^s \LAM_{I^j}(E^j_{\bull})$
 is a transverse intersection at
$M$ and hence $h^1(\mb^*)=0$. Thus,  $h^0(\mb^*)=0$ as well and  we get the following:
\begin{itemize}
\item $h^0(\ma^*)=0$. Hence, by Proposition  ~\ref{morning} (1)-(2), Identity
 ~\eqref{100} and Corollary ~\ref{key2} applied to $V=M^\perp/M$,  condition $(\beta_2)$
holds.
\item  $h^0(\mc^*)=0$ which implies condition $(\beta_3)$ (using Lemma ~\ref{old3}, Proposition
~\ref{morning} (3) and the condition $(\beta_1)$).
\end{itemize}

\noindent
$(\beta)\Rightarrow (\alpha)$:
 Assume $(\beta)$. Find
 a $s$-tuple $\me_{\bull}=(E^1_{\bull},\dots,E^s_{\bull})$ of isotropic flags
on $V$ such that $M$ lies in the intersection $\Omega=\cap_{j=1}^s
\LAM_{I^j}(E^j_{\bull})$. The parameter space of such $\me_{\bull}$ is
irreducible (Lemma ~\ref{attract});  pick a generic
such $\me_{\bull}$. It follows that the induced flags on $M$ and
$M^{\perp}/M$ are generic. Therefore, our assumptions $(\beta_1), (\beta_2)$,
Identity  ~\eqref{100},  Corollary ~\ref{key2} applied to $V=M^\perp/M$ and
Proposition
~\ref{morning} (1)-(2)  imply that
$h^0(\ma^*)=0$. Moreover, our assumption $(\beta_3)$,  Proposition
~\ref{morning} (3) and Lemma ~\ref{old3} imply that $h^0(\mc^*)=0$.
 Hence, $h^0(\mb^*)=0$ and thus $h^1(\mb^*)=0$. Therefore,
$\Omega$ is a transverse intersection at $M$, and hence $(\alpha)$
holds by the equivalence of ($\gamma_1$) and ($\gamma_2$). This proves the theorem.\qed

\section{Horn's problem for odd orthogonal
  groups}\label{finalproblem}

In this section we extend the results from the last section
 to the odd orthogonal groups. The proofs are similar, so we will be brief.
 Refer to the Section 2.3 for various notation.  
  Fix $r\leq n$ and let $J^1,\dots,J^s\in \FS'(r,2n+1)$ be such that
\begin{equation}\label{postman1}
\sum_{j=1}^s\codim({\Psi}_{{J^j}})=\dim \OG(r,2n+1).
\end{equation}

Let $a_j$ be the largest integer $\leq n$ in $J^j\sqcup \bar{J'}^j$ and set
 $b_j:=2n+2-a_j$. Let $\beta_j$ be the
order preserving  bijection of $(J^j\sqcup \bar{J'}^j)\setminus \{a_j,b_j\}$
with $[2r-2]$ and let $J_o^j$ be the subset $\beta_j(J^j\setminus
 \{a_j,b_j\})$ of $[2r-2]$
(of cardinality $r-1$). For any $a\in [r]$ and $j\in [s]$, let
$\lambda^j_a:=|i_a^j\geq \tilde{J}^j|$ and $\mu_a^j:=2n+1-2r-\lambda^j_a$,
where $J^j= \{i^j_1 <\cdots
<i^j_{r}\}$ and $ \tilde{J}^j:=[2n+1]\setminus (J^j\sqcup \bar{J'}^j)$.
\begin{theorem}\label{challenge1}
Under the assumption  ~\eqref{postman1}, the following conditions ($\alpha$) and
$(\beta)$ are equivalent.
\begin{enumerate}\item[$(\alpha)$] $\prod_{j=1}^s[\bar{\Psi}_{{J^j}}]\neq 0\in H^{*}(\OG(r,2n+1), \odot_0)$.
\item[$(\beta)$]
\begin{enumerate}
\item[$(\beta_1)$]
$\dim \OG(r,2r)= \sum_{j=1}^s \coo\wedge^2(J^j),$  where $\OG(r,2r)$ is the orthogonal
Grassmannian of isotropic $r$-dimensional subspaces in the $2r$-dimensional space with a
symmetric nondegenerate form.
\item[$(\beta_2)$]
$\bigl(V_{\mu^1}\tensor\cdots\tensor
V_{\mu^s}\bigr)^{SL(r)}\neq 0$.
\item[$(\beta_3)$] $\prod_{j=1}^s[\bar{\LAM}_{{{J^j_o}}}]\neq 0\in
H^{*}(\IG(r-1,2r-2)).$
\end{enumerate}
\end{enumerate}
\end{theorem}
Similar to Remark ~\ref{purbhoo}, we have the following:
\begin{remark}
Assume that the condition $(\beta_1)$ is satisfied. Then, the  condition
 $(\beta_2)$ is equivalent to the following condition.

  For any $1\leq d\leq r$ and subsets ${\sm}^1,\dots,{\sm}^s$ of $[r]$ each of cardinality $d$
such that the product $\prod_{j=1}^s [\OMC_{{\sm}^j}]\neq 0\in
H^*(\Gr(d,r))$, the following inequality holds
$$\sum_{j=1}^s\sum_{a\in {\sm}^j} \mu^j_a \leq d(2n+1-2r).$$
\end{remark}

Before we can complete the proof of the theorem, we need an interplay 
between the following three homogeneous spaces.

\subsection{Fundamental triple of classical homogenous spaces}
As is well known, the following three homogenous spaces  have `essentially' equivalent
 intersection theory.
\begin{enumerate}
\item The isotropic Grassmannian $\IG(r-1,2r-2)$.
\item The orthogonal Grassmannian $\OG(r-1,2r-1)$.
\item The orthogonal Grassmannians $\OG^\pm(r,2r)$, where  $\OG^\pm(r,2r)$ are the two
components of the orthogonal Grassmannian  $\OG(r,2r)$.
\end{enumerate}

 The relation between  (1) and (2)  is discussed in the appendix. 
The spaces (2) and (3)
 are actually homeomorphic.
Let $V$ be a $2r$-dimensional vector space with a symmetric
nondegenerate bilinear form. Fix a subspace  $T$
of $V$ of dimension $2r-1$, such that the quadratic
form is nondegenerate on $T$. Then, there are homeomorphisms
$$\varphi^\pm:\OG^\pm (r,V)\leto{\sim} \OG(r-1,T),$$
which send $N$ to $N\cap T$. Notice that $N$ can not be a subspace of
$T$ (because $T$ does not have isotropic subspaces of dimension
$r$).

Fix one component and denote it by   $\OG^+(r,2r)$. Then, 
as in [BL, \S3.5], the
Schubert cells in $\OG^+(r,2r)$ are parameterized by the  subsets $I\subset
[2r]$ of cardinality $r$ satisfying:
\begin{itemize}
\item $I\cap\bar{I}=\emptyset$, where $\bar{I}:=\{2r+1-a|a\in I\}$.
\item The number of elements in $I$ that are less than or equal to
$r$ has the same parity as $r$.
\end{itemize}
One easily sees that the above set is in bijection with
$\FS(r-1,2r-2)$, the indexing set for the Schubert cells in
$\IG(r-1,2r-2)$. This bijection takes $I$ to  $\beta(I\setminus\{r,r+1\}),$
 where
$\beta$ is the  order
preserving bijection $[2r]\setminus \{r,r+1\}\leto{\sim} [2r-2]$. By Theorem
~\ref{grain} and the above homeomorphism $\varphi^+$, this
bijection  preserves non-zeroness of intersection
numbers.

\subsection{Proof of Theorem ~\ref{challenge1}}\label{thefinalproblem}
Under the assumption  ~\eqref{postman1}, the following two conditions are equivalent.
\begin{enumerate}

\item[$(\gamma_1)$] $\prod_{j=1}^s[\bar{\Psi}_{{J^j}}]\neq 0\in
H^{*}(\OG(r,2n+1),\odot_0).$
\item[$(\gamma_2)$]  $\prod_{j=1}^s[\bar{\Psi}_{{J^j}}]\neq 0\in
H^{*}(\OG(r,2n+1))$ and the condition $(\beta_1)$ is satisfied.
\end{enumerate}

A straightforward
reworking of the proof of Theorem ~\ref{challenge} yields Theorem
~\ref{challenge1} with condition ($\beta_3$) replaced by
the following:

\begin{enumerate}
\item[$(\beta_3')$] For generic complete flags  $F^1_{\bull}, \dots, F^s_{\bull}$ on
a $r$-dimensional vector space $M$, the vector space
\begin{equation}\label{vss34}
\{\gamma\in \wedge^2 M^*\mid \gamma(F^j_a, F^j_{t^j_a})=0,\ a\in[r],
j\in [s]\}
\end{equation}
 is of the expected dimension
$$r(r-1)/2-\sum_{j=1}^s\co\wedge^2(J^j),$$
where $J^j:=\{i^j_1<\dots< i^j_r\}$ and $t^j_a:=|\bar{J'}^j\geq i^j_a|$.
\end{enumerate}

 In the following lemma, we will consider only those isotropic
flags $F_{\bull}$ on $V$ such that $F_{r}\in \OG^+(r,2r)$. We have the following analogue of
Proposition ~\ref{morning} (3) for the Grassmannian $\OG^+(r,2r)$.
\begin{lemma}
The tangent space at a point $P$ to the shifted Schubert cell in $\OG^+(r,2r)$ parameterized by
$I=\{i_1< \dots< i_r\}$ and corresponding to an isotropic flag $F_{\bull}$ on $V$
 is given by
\begin{equation}\label{105}
\{\gamma\in \wedge^2 P^*\mid \gamma(F(P)_a, F(P)_{t_a})=0,\forall a\in [r]\},
\end{equation}
where $t_a:=|\bar{I}\geq i_a|$ and (as earlier) $F(P)_{\bull}$ is the filtration on $P$
induced from that of $F_{\bull}$.
\end{lemma}
Let $s_r$ be the permutation in the symmetric group $S_{2r}$ which transposes $r$ and $r+1$.
Then, observe that in the above lemma, the space (\ref{105}) corresponding to $I$
is the same as that for $s_r(I)$.

Using an obvious analogue of Lemma
~\ref{old3} (with $\IG(r,2r)$ replaced by $\OG^+(r,2r)$) and the
relation between  structure constants in $\OG^+(r,2r)$ and
$\IG(r-1,2r-2)$ as given in Subsection 9.1, we conclude that the conditions $(\beta_3)$ and
$(\beta_3')$  are the same. This proves Theorem ~\ref{challenge1}. \qed

\section{Appendix. Relation between intersection theory of homogenous spaces
for $\Sp(2n)$ and $\SO(2n+1)$}\label{one1}

For completeness, we include a prof of the following well-known Theorem ~
\ref{grain} (cf., e.g., [BS, pages 2674-75]).
As in Section 2, there is a canonical Weyl group equivariant identification between the Cartan
subalgebras $\frh^C$ and $\frh^B$ (and also the Weyl groups $W^C$ and $W^B$)
of  $\Sp(2n)$ and $\SO(2n+1)$ respectively. We will make these identifications
and denote  them by $\frh$ and $W$.
Let $A_i^C$ (respectively,  $A_i^B$) be the BGG operators for $\Sp(2n)$
(respectively,  $\SO(2n+1)$) acting on $S(\frh^*)=
\Bbb C[\epsilon_1,\dots, \epsilon_n]$ corresponding to the simple reflections $s_i$, where $\{\epsilon_1,\dots, \epsilon_n\}$ is the basis of $\frh^*$ as in
[Bo, Planche II and III].
\begin{lemma}
\begin{align*}
A_i^B &= A_i^C, \qquad\qquad\; 1\leq i<n \\
\frac{1}{2} A_n^B &= A_n^C.
  \end{align*}
\end{lemma}
For any $w\in W$ with reduced decomposition $w=s_{i_1}\cdots s_{i_p}$, define
$A_w^C=A^C_{i_1}\cdots A^C_{i_p}$  and similarly define $A_w^B$. (They do not depend upon the choice of the reduced decomposition.)
Consider the isomorphism
\[\phi:H^*\Bigl( \SO(2n+1) /B^B,\Bbb C\Bigr) \to  H^*\Bigl( \Sp(2n) /B^C,\Bbb C\Bigr),\]
induced
from the
Borel isomorphism via the identity map
\[ \frac{\Bbb C[\eps_1,\dots ,\eps_n ]}{I^B} \overset{I}{\simto}
   \frac{\Bbb C[\eps_1,\dots ,\eps_n ]}{I^C},\]
where $I^B=I^C$ is the ideal generated by the positive degree
$W$-invariants in $\Bbb C[\eps_1,\dots ,\eps_n ]$. Let
$\{[\bar{\Lambda}_w(C)]\}_{w\in W} \subset H^*(\Sp(2n)/B^C)$ be the
Schubert basis and similarly for  $\{[\bar{\Lambda}_w(B)]\}_{w\in
W}$. Let $p_e^C \in \Bbb C[\epsilon_1,\dots, \epsilon_n]$ be a
representative of $[\bar{\Lambda}_e(C)]$ and similarly $p^B_e$ for
$[\bar{\Lambda}_e(B)]$. Then, by [BGG], we can take
 \begin{align*}
p^B_{e} &= \frac{\bigl(\prod^n_{i=1} \eps_i\bigr)\bigl( \prod_{1\leq i<j\leq n}
( \eps_i^2 -\eps_j^2)\bigr)}{|W|},\,\,\,\text{and} \\
p^C_{e} &= \frac{2^n\,\bigl( \prod^n_{i=1} \eps_i\bigr)\bigl( \prod_{1\leq
i<j\leq n} ( \eps_i^2 -\eps_j^2)\bigr)}{|W|}.
 \end{align*}
Thus,
 \[p^C_{e} = 2^{n}\, p^B_{e}.\]
By loc cit., $p_w^C:=A^C_{w^{-1}}p_e^C$ represents the class
$[\bar{\Lambda}_w(C)]$ and similarly  $p_w^B:=A^B_{w^{-1}}p_e^B$ represents the class
$[\bar{\Lambda}_w(B)]$. Hence
\[p_w^C=2^{n-\mu(w)}p_w^B,\]
where (as in Section 3) $\mu(w)$ represents the number of times the simple
reflection $s_n$ appears in any reduced decomposition of $w$.
Thus, we have the following:

  \begin{theorem}\label{grain}
The algebra
isomorphism $\phi$ satisfies:
\begin{equation*}
\phi\bigl([\bar{\Lambda}_w(B)]\bigr) = 2^{\mu(w)-n}\,
[\bar{\Lambda}_w(C)], \,\,\text{for any}\,\, w\in W.
\end{equation*}
  \end{theorem}


\begin{thebibliography}{BaBE}
\addcontentsline{toc}{section}{Bibliography}


\bibitem[B$_1$]{gh}
P. Belkale,
 {\em Geometric proofs of Horn and saturation
conjectures}, J. Algebraic Geom.  15  (2006), 133--173.

\bibitem[B$_2$]{invariant}
P. Belkale, {\em Invariant theory of $\GL(n)$ and intersection
theory of Grasmannians}, International Math. Research Notices, Vol.
2004, no. 69, 3709--3721.

\bibitem[BK]{BK}
P. Belkale and S. Kumar,
{\em Eigenvalue problem and a new product in cohomology of flag varieties}, Invent. Math.  166  (2006), 185--228.



\bibitem[BeSj]{BerensteinSjamaar}
A. Berenstein and R. Sjamaar, {\em Coadjoint orbits, moment
polytopes, and the Hilbert-Mumford criterion},  Journ. Amer. Math.
Soc. 13 (2000),   433--466.


\bibitem[BS]{Bergeronsottile}
N. Bergeron and F. Sottile, {\em A Pieri type formula for isotropic 
flag manifolds},  Trans. AMS
354 (2002),   2659--2705.

\bibitem[BGG]{Bernstein}
I. Bernstein, I. Gelfand and  S. Gelfand, {\em Schubert cells and the cohomology of spaces $G/P$},  Russian Math.
Surveys 28 (1973),   1--26.


\bibitem[Bo]{Bourbaki}
N. Bourbaki, {\em Groupes et Alg\'ebres de Lie, Chapitres 4,5 et 6},
Hermann, Paris (1968).

\bibitem[BL]{lakshmibai}
S. Billey and V. Lakshmibai, {\em Singular Loci of Schubert
Varieties}, Progress in mathematics, Vol. 182 (2000).

\bibitem[F$_1$]{fulton2}
W. Fulton, {\em Young Tableaux}, Cambridge University Press (1997).

\bibitem[F$_2$]{fulton1}
W. Fulton, {\em Intersection Theory}, Second Edition, Springer (1998).

\bibitem[F$_3$]{ful4}
W. Fulton.
\newblock Eigenvalues of majorized Hermitian matrices and Littlewood-Richardson coefficients.
Lin. Alg. Appl 319 (2000), 23-36.

\bibitem[F$_4$]{ful}
W. Fulton.
\newblock Eigenvalues, invariant factors, highest weights, and
              Schubert calculus.
\newblock Bull. Amer. Math. Soc. (N.S.) 37 (2000), 209--249.
\bibitem[KM]{KM}
M.  Kapovich and  J. J. Millson, {\em A path model for geodesics in
Euclidean buildings and its applications to representation theory},
preprint, math.RT/0411182.
\bibitem[K]{Klyachko}
A. Klyachko, {\em Stable bundles, representation theory and Hermitian
operators}, Selecta Math. 4(1998), 419--445.
\bibitem[KT]{KT}
A. Knutson and  T. Tao. {\em The Honeycomb model of ${\rm GL}\sb n(C)$
tensor products I: Proof of the Saturation conjecture},  J. Amer.
Math. Soc. 12 (1999), 1055--1090.
\bibitem[KuLM]{KuLM}
S.  Kumar, B. Leeb and J. J. Millson,  {\em The generalized triangle
inequalities for rank 3 symmetric spaces of noncompact type},   Contemp. Math.
 332 (2003),   171--195.
\bibitem[KS]{ks}
S. Kumar and  J. Stembridge, {\em Special isogenies and tensor
product multiplicities}, Preprint (2007).


\bibitem[PS]{ps}
K. Purbhoo and  F. Sottile, {\em The recursive nature of the cominiscule
Schubert calculus}, Preprint.
\bibitem[R]{ed} E. Richmond, {\em Horn recursion for a new product in the cohomology of partial flag varieties
SLn/P}, Preprint.
\bibitem[S]{sch}    A. Schofield, {\em General
representations of quivers}, Proc. London Math. Soc. 65 (1992),
46--64.


\bibitem[Sj]{Sjamaar}
 R. Sjamaar, {\em Convexity properties of the moment
mapping re-examined},  Adv. Math.
 138 (1998),   46--91.
\end{thebibliography}
\end{document}